\documentclass[11pt, oneside]{article}
\usepackage[english, french]{babel}
\usepackage[latin1]{inputenc}
\usepackage{geometry}
\usepackage{graphicx}		
\usepackage{amsmath}
\usepackage{amsfonts} 
\usepackage{amssymb}
\usepackage{amsthm}
\usepackage[T1]{fontenc}
\usepackage{hyperref}
\usepackage{stmaryrd}
\usepackage{dsfont}
\usepackage[all]{xy}
\usepackage{sectsty}
\usepackage{multicol}
\usepackage{eucal}
\usepackage[numbers]{natbib}
\usepackage{tikz}
\usepackage{yhmath}

\sectionfont{\centering}

\def\N{\mathbb{N}}
\def\Z{\mathbb{Z}}
\def\c{\mathbb{C}}
\def\Q{\mathbb{Q}}

\def\Du{\mathbb{D}}

\DeclareMathOperator{\Tr}{Tr}
\DeclareMathOperator{\Frac}{Frac}
\DeclareMathOperator{\Fil}{Fil}
\DeclareMathOperator{\fil}{Fil}
\DeclareMathOperator{\id}{id}

\DeclareMathOperator{\Car}{Car}

\DeclareMathOperator{\gr}{gr}
\DeclareMathOperator{\ann}{Ann}
\DeclareMathOperator{\cc}{CC}
\DeclareMathOperator{\supp}{Supp}

\DeclareMathOperator{\Exp}{Exp}
\DeclareMathOperator{\codim}{codim}

\newtheorem{theorem}{Théoreme}[section]
\newtheorem{lemma}[theorem]{Lemme}
\newtheorem{prop}[theorem]{Proposition}
\newtheorem{cor}[theorem]{Corollaire}
\newtheorem{definition}[theorem]{Définition}
\newtheorem{remark}[theorem]{Remarque}
\theoremstyle{definition}
\newtheorem{example}[theorem]{Exemple}

\def\V{\mathcal{V}}
\def\I{\mathcal{I}}
\def\J{\mathcal{J}}
\def\A{\widehat{\mathbb{A}}_K ^{1, \mathrm{an}} }
\def\D{\mathcal{D}}
\def\Dk{\widehat{\mathcal{D}}^{(0)}_{\X, k}}
\def\Dks{\mathcal{D}_{X, k}}
\def\Dx{\mathcal{D}_{X, k, x}}
\def\Dkq{\widehat{\mathcal{D}}^{(0)}_{\X, k, \Q}}
\def\X{\mathfrak{X}}
\def\m{\mathfrak{m}}
\def\Spf{\mathrm{Spf}\,}

\def\Spec{\mathrm{Spec}\,}
\def\Spm{\mathrm{Spm}\,}
\def\O{\mathcal{O}}
\def\iso{\overset{\sim}{\longrightarrow}}
\def\d{\mathrm{d}}
\def\sp{\mathrm{sp}}
\def\Nb{\overline{N}_k}
\def\Na{\overline{N}}
\def\E{\mathcal{E}}
\def\F{\mathcal{F}}
\def\Eo{\mathcal{E}^\circ}
\def\M{\mathcal{M}}
\def\Nn{\mathcal{N}}
\def\L{\mathcal{L}}
\def\Ext{\mathcal{E}xt}

\def\Di{\mathcal{D}_{\X, \infty}}
\def\H{\mathcal{H}(\Di)}

\title{$\Dkq$-modules holonomes sur une courbe formelle}
\author{Raoul Hallopeau}
\date{}

\begin{document}

\maketitle

\selectlanguage{english}
\begin{abstract}
Let $\mathfrak{X}$ be a formal smooth curve over a complete discrete valuation ring $\mathcal{V}$ of mixed characteristic $(0 , p)$.
Let $\widehat{\mathcal{D}}^{(0)}_{\mathfrak{X}, \mathbb{Q}}$ be the sheaf of crystalline differential operators of level 0 (i.e. generated by the derivations).
In this situation, Garnier proved that holonomic $\widehat{\mathcal{D}}^{(0)}_{\mathfrak{X}, \mathbb{Q}}$-modules as defined by Berthelot have finite length.
In this article, we address this question for the sheaves $\widehat{\mathcal{D}}^{(0)}_{\mathfrak{X}, k , \mathbb{Q}}$ of congruence 
level $k$ defined by Christine Huyghe, Tobias Schmidt and Matthias Strauch.
Using the same strategy as Garnier, we prove that holonomic $\widehat{\mathcal{D}}^{(0)}_{\mathfrak{X}, k , \mathbb{Q}}$-modules have finite length.
We finally give an application to coadmissible modules by proving that coadmissible modules with integrable connection over curves have finite length.
\end{abstract}

\selectlanguage{french}
\tableofcontents

\section{Introduction}

Un enjeu principal dans la théorie des $\D$-modules arithmétiques consiste à y généraliser la notion de $\D$-modules holonomes.
Une catégorie constituée de tels modules permettra par exemple d'étudier les cohomologies cristalline et rigide.
Des catégories de $\D$-modules arithmétiques holonomes avec Frobenius ont été construites par Berthelot et Caro de deux manières différentes. 
Berthelot a défini une variété caractéristique comme une sous-variété fermée du fibré cotangent de l'espace de base, ce qui lui a permis de définir la notion d'holonomie comme dans le cas complexe.
Caro a quant à lui construit une catégorie de $\D$-modules surholonomes munis  d'un Frobenius a priori stable par les six opérations et a montré que cette catégorie coïncide avec la catégorie de Berthelot  dans le cas des schémas quasi-projectifs.

Christine Huyghe, Tobias Schmidt et Matthias Strauch ont introduit dans \cite{huyghe} un faisceau $\Di = \varprojlim_k \Dkq$ d'opérateurs différentiels surconvergeants obtenu en rajoutant des niveaux de congruence $k \in \N$ aux faisceaux de Berthelot.
Les modules à considérer sur $\Di$ ne sont plus les $\Di$-modules cohérents mais les $\Di$-modules coadmissibles.
Nous ne disposons pas de bonne notion d'holonomie pour ces modules.
L'article \cite{ABW2} de Ardakov-Bode-Wadsley  introduit une catégorie de modules faiblement holonomes de $\D$-modules coadmissibles sur un espace analytique rigide en utilisant la caractérisation classique des modules holonomes.
Cependant, cette catégorie demeure trop grosse, les $\D$-modules faiblement holonomes ne sont par exemple pas tous de longueur finie.
Une autre approche possible consiste à définir une variété caractéristique pour les $\D$-modules coadmissibles.
Une première étape afin d'obtenir une catégorie de $\Di = \varprojlim_k \Dkq$-modules holonomes est de commencer par définir une bonne catégorie de $\Dk$-modules holonomes pour un niveau de congruence $k$ fixé.
C'est l'objectif de cet article dans le cas où $\X$ est une courbe formelle.
Nous introduirons dans un autre article une variété caractéristique pour les $\Di$-modules coadmissibles afin d'obtenir une notion d'holonomie dans ce contexte.

\vspace{0.4cm}
Expliquons maintenant le cadre et les résultats de cet article.
Soit $\V$ un anneau complet de valuation discrète de caractéristique mixte $(0,p)$ et $K = \Frac(\V)$ son corps des fractions.
Nous fixons un $\V$-schéma formel lisse $\X$ dont l'idéal de définition est engendré par une uniformisante $\varpi$ de $\V$.
Nous considérons le faisceau $\Dk$ introduit dans l'article \cite{huyghe} de Huyghe-Schmidt-Strauch.
Il s'agit d'un faisceau de sous-algèbres du faisceau $\widehat{\mathcal{D}}^{(0)}_\X$ des opérateurs différentiels cristallins obtenu en rajoutant un paramètre $k \in \N$ appelé niveau de congruence.
Soit $U$ un ouvert affine de $\X$ sur lequel on dispose d'un système de coordonnées étales $(x_1, \dots , x_d)$. Si $\partial_1 , \dots , \partial_d$ sont les dérivations associées, alors
\[ \Dk(U) = \left\{ \sum_{\alpha \in \N^d} a_\alpha \cdot \partial_1^{\alpha_1} \dots \partial_d^{\alpha_d},~~ a_\alpha \in \O_\X(U) ~~\mathrm{tels~ que} ~~a_\alpha \cdot \varpi^{-k | \alpha|} \underset{ | \alpha| \to \infty}{\longrightarrow} 0 \right\}. \]
Pour $k = 0$, nous retrouvons le faisceau de Berthelot $\widehat{\mathcal{D}}^{(0)}_\X$ de niveau $m = 0$. Notons $\Dkq = \Dk \otimes_\V K$.
Pour $k'\geq k$, nous disposons d'une inclusion $\widehat{\mathcal{D}}^{(0)}_{\X , k' , \Q}(U) \subset \Dkq(U)$.
Ces inclusions locales induisent un morphisme de transition $\widehat{\mathcal{D}}^{(0)}_{\X , k+1 , \Q} \to \Dkq$.
On note $\Di = \varprojlim_k \Dkq$ le faisceau limite projective des faisceaux $\Dkq$.

\vspace{0.4cm}

Rajouter un niveau de congruence $k$ au faisceau $\widehat{\mathcal{D}}^{(0)}_{\X , \Q}$ des opérateurs différentiels cristallins est très intéressant pour plusieurs raisons brièvement décrites ci-dessous.
Tout d'abord, les faisceaux $\Dkq$ pour les différents niveaux de congruences $k$ interviennent naturellement pour résoudre certaines questions données par exemple dans l'article \cite{huyghe2} de Christine Huyghe, Tobias Schmidt et Matthias Strauch.
Les faisceaux $\Dkq$ apparaissent dans l'étude de représentations localement analytiques de groupes de Lie $p$-adique. Ils s'avèrent aussi utiles pour regarder des isocristaux surconvergents dans le cas ramifié.

Par ailleurs, d'un point de vue conceptuel, les faisceaux $\Dkq$ sont pertinents.
En effet, nous pouvons associer aux éléments de $\widehat{\mathcal{D}}^{(0)}_{\X , \Q}$ des fonctions analytiques sur le fibré cotangent $T^*\X$ de $\X$ convergents sur une bande horizontale de l'espace analytique rigide associé, le fibré cotangent rigide $T^* \X_K$.
Les opérateurs de $\Dkq$ définissent des fonctions sur $T^*\X_K$ convergents sur un domaine grossissant avec $k$. Ces régions recouvrent $T^*\X_K$ lorsque $k$ tend vers l'infini.
Plus précisément, soit $(x_1 , \dots , x _n , \xi_d , \dots , \xi_d)$ un système de coordonnées locales sur $T^*U$ associée aux coordonnées étales de départ sur $U$.
Nous pouvons associer à tout opérateur $P = \sum_{\alpha \in \N^d} a_\alpha(x) \cdot \partial_1^{\alpha_1} \dots \partial_d^{\alpha_d}$ de l'algèbre $\widehat{\mathcal{D}}^{(0)}_{\X , \Q}(U)$ un élément $P(x , \xi) = \sum_{\alpha \in \N^d} a_\alpha(x) \cdot \xi_1^{\alpha_1} \dots \xi_d^{\alpha_d}$ du fibré cotangent rigide $T^*U_K$.
La fonction $P(x , \xi)$ converge sur la bande horizontale $\{ | \xi_1 | \leq 1 , \dots , |\xi_d| \leq 1 \}$ de $T^*U_K$.
Un opérateur différentiel $P$ de $\Dkq(U)$ définit une fonction analytique $P(x , \xi)$ convergente sur la bande horizontale $\{ | \xi_1 | \leq p^k , \dots , |\xi_d| \leq p^k \}$.
Ainsi, les opérateurs différentiels de l'algèbre $\Di(U)  = \varprojlim_k \Dkq(U) = \bigcap_k \Dkq(U)$ induisent des fonctions analytiques entières sur le fibré cotangent rigide $T^*\X_K$.

\vspace{0.4cm}

Lorsque $X$ est une variété complexe lisse, la variété caractéristique d'un $\D_X$-module cohérent $M$ non nul est une sous-variété involutive du fibré cotangent $T^*X$.
La preuve de ce résultat repose sur la caractéristique nulle de $\c$. En particulier, une composante irréductible de $\Car M$ a une dimension supérieure à celle de $X$.
Le module $M$ est appelé holonome si $\dim \Car M = \dim X$. La minimalité des dimensions des composantes irréductibles de la variété caractéristique $\Car M$ implique que $M$ est de longueur finie.

Soit maintenant $\E$ un $\Dk$-module à gauche cohérent. Sa variété caractéristique $\Car \E$ est définie en généralisant la construction de Berthelot pour un niveau de congruence $k$ comme suit.
Notons $\kappa$ le corps résiduel de $\V$ et $X = \X \times_\V \Spec \kappa$ la fibre spéciale de $X$.
La réduction $E = \E \otimes_\V \kappa$ modulo $\varpi$ de $\E$ est un $\Dks$-module cohérent, où $\Dks = \Dk \otimes_\V \kappa$ est un faisceau sur $X$.
Les opérateurs différentiels de $\Dks$ étant finis, on munit $\Dks$ de la filtration donnée par l'ordre des opérateurs différentiels.
Classiquement, la variété caractéristique de $E$ est construite comme une sous-variété fermée du fibré cotangent $T^* X $ de $X$.
La variété caractéristique de $\E$ est par définition celle de $E$.
Un $\Dk$-module cohérent dont la variété caractéristique est de dimension au plus la dimension de $\X$ est appelé module holonome.
Cependant, les méthodes utilisées pour une variété complexe ne s'appliquent plus puisque la caractéristique de $\kappa$ est positive (la fibre spéciale $X$ de $\X$ est un $\kappa$-schéma).
Le fait que ces modules soient de longueur finie n'est pas connu en général.

\vspace{0.4cm}

Laurent Garnier a démontré dans \cite{garnier} que les $\widehat{\mathcal{D}}^{(0)}_{\X , \Q}$-modules holonomes sont de longueur finie lorsque $\X$ est une courbe formelle.
Nous généralisons dans cet article ce résultat à un niveau de congruence $k \in \N$ quelconque toujours pour une courbe formelle $\X$.
Nous adaptons les constructions et les preuves de Laurent Garnier pour les $\Dkq$-modules cohérents dans les sections deux et trois.

La partie \ref{partie3} commence par quelques rappels et propriétés sur les faisceaux $\O_{\X , \Q}$ et $\Dkq$.
Nous introduisons dans la section \ref{partie4} les variétés caractéristiques des $\Dkq$-modules cohérents.
Nous expliquons dans la partie \ref{partiedx} qu'il est suffisant d'étudier les variétés caractéristiques des quotients $\Dx / I$ de $\Dx := \Dks \otimes_\kappa \O_{X,x}$.
Nous démontrons ensuite dans \ref{partie2.3} l'inégalité de Bernstein : les composantes irréductibles de la variété caractéristique d'un $\Dkq$-module cohérent non nul sont de dimension au moins un.
Un $\Dkq$-module cohérent $\E$ est dit holonome si $\dim \Car \E \leq \dim X = 1$.
Nous prouvons enfin dans la partie \ref{partie2.4} le résultat suivant.

\begin{prop}
Soit $\E$ un $\Dkq$-module cohérent. Les énoncés suivants sont équivalents :
\begin{enumerate}
\item
$\E$ est holonome ;
\item
$\E$ est localement de la forme $\widehat{\mathcal{D}}^{(0)}_{U, k, \Q} /\I$ pour un idéal cohérent $\I \neq 0$  ;
\item
$\E$ est de longueur finie ;
\item
$\E$ est de torsion ;
\item
$\Ext_{\Dkq}^d(\M , \Dkq) = 0$ pour tout entier $d \neq 1$ ;
\item
il existe un ouvert non vide $U$ de $\X$ tel que $\E_{|U}$ soit un $\O_{\X , \Q |U}$-module libre de rang fini. Autrement dit, $\E_{|U}$ est un module à connexion intégrable.
\end{enumerate}
\end{prop}

Considérons maintenant le faisceau $\Di = \varprojlim_k \Dkq$.
Dans la section \ref{section4}, nous appliquons les résultats précédents aux $\Di$-modules coadmissibles, c'est à dire aux $\Di$-modules isomorphes à une limite projective de $\Dkq$-modules cohérents $\M_k$ ayant de bonnes propriétés de transitions entre les différents niveaux de congruences.
En particulier, nous construisons une catégorie abélienne formée de $\Di$-modules coadmissibles de longueur finie.
Elle est constituée des modules coadmissibles $\M \simeq \varprojlim_k \M_k$ vérifiant les deux points suivants.
\begin{enumerate}
\item
Il existe un rang $k_0$ tel que pour tout $k \geq k_0$, $\M_k$ est un $\Dkq$-module holonome.
\item
La limite supérieure pour $k \geq k_0$ des multiplicités des modules $\M_k$ est finie.
\end{enumerate}
Nous montrons que cette catégorie n'est pas triviale. En effet, elle contient les $\Di$-modules coadmissibles de la forme $\Di /P$ dès que $P$ un opérateur différentiel fini de $\Di$.
Nous montrons enfin que les modules coadmissibles à connexion intégrable appartiennent à cette catégorie.
Les modules à connexion intégrable sont donc de longueur finie.

\subsection*{Notations}

\begin{enumerate}
\item[$\bullet$]
$\V$ est un anneau complet de valuation discrète de caractéristique mixte $(0,p)$, d'idéal maximal $\m$ et de corps résiduel $\kappa$ supposé parfait. On note $| \cdot |$ la valeur absolue normalisée de $\V$, $\varpi$ une uniformisante et $K =\Frac(\V)$ son corps des fractions.
\item[$\bullet$]
$X$ est une courbe sur $\kappa$ lisse connexe quasi-compacte et $x \in X$ est un point fermé donné.
\item[$\bullet$]
$\X$ est un $\V$-schéma formel lisse localement de type fini relevant $X$ d'idéal de définition engendré par l'uniformisante $\varpi$.
\item[$\bullet$]
$\X_K$ est l'espace analytique rigide associé à $\X$.
\item[$\bullet$]
$t$ est un relèvement local sur $\O_\X$ d'une uniformisante en $x$ ($\O_{X , x}$ est un anneau de valuation discrète puisque $X$ est une courbe). Alors $\d t$ est une base de $\Omega^1 _{\X,x}$. On note $\partial$ la dérivation associée.

\item[$\bullet$]
$U$ est un ouvert affine de $\X$ contenant $x$ sur lequel on dispose d'une coordonnée locale.
\item[$\bullet$]
Soit $f\in \Gamma(U , \O_{\X, \Q}) \backslash \{0\}$ et $r$ tel que $f_1 : = \varpi^r f \in \Gamma(U , \O_\X) \backslash \Gamma(U , \m \cdot \O_\X)$. On note $U_{\{ f \}}\subset U$ l'ouvert sur lequel $f_1$ est inversible. On remarquera que $ U_{\{ f \} } \cup \{ x \} = U \backslash \{ V(\bar{f}_1) - \{x\} \}$ (où $\bar{f}_1$ est la réduction de $f_1$ modulo $\m$) est un ouvert puisque $\bar{f}_1$ n'a qu'un nombre fini de zéros.
\item[$\bullet$]
Sauf mention contraire, les idéaux et les modules considérés seront tous à gauche.
\end{enumerate}

\section{Propriétés du faisceau $\Dkq$}\label{partie3}

On adapte dans cette section la seconde partie de l'article \cite{garnier} de Laurent Garnier à un indice de congruence $k \in \N$.
On munit l'algèbre $\Dkq(U)$ d'une norme complète multiplicative puis on montre la simplicité de cette dernière dans la partie \ref{partie1.2}.
On énonce ensuite quelques théorèmes de division sur $\Dkq$ dans la section \ref{partie1.3}
On termine enfin par quelques rappels et quelques propriétés sur les bases de division d'un idéal cohérent de $\Dkq$ dans la dernière sous-partie \ref{partie1.4}.

\subsection{Rappels sur la norme spectrale de $\O_{\X , \Q}$}

On redonne ici la définition d'une algèbre affinoïde et de sa norme spectrale.
Puis on rappelle quelques résultats utiles de la première partie de l'article \cite{garnier} de Garnier. On pourra s'y référer pour les preuves des lemmes énoncés.

\vspace{0.4cm}

On note $T_n(\V) = \V \langle T_1 , \dots , T_n \rangle$ l'algèbre de Tate sur $\V$ à $n$-variables.
Si $T^\alpha = T_1^{\alpha_1}  \dots T_n^{\alpha_n}$ et $|\alpha| = \alpha_1 + \dots + \alpha_n$, alors
\[ T_n(\V) = \left \{ f(T) = \sum_{\alpha\in \N^n} c_\alpha \cdot T^\alpha,~~ |c_\alpha| \underset{|\alpha| \to \infty}{\to} 0 \right\}. \]
On munit $T_n(\V)$ de la norme de Gauss définie par $|f| = \max\{|c_\alpha|\}$. C'est une valuation et $T_n(\V)$ est le complété de $\V[ T_1 , \dots , T_n]$ pour cette valuation. En particulier, $T_n(\V)$ est une $\V$-algèbre de Banach.
Elle est de plus noetherienne et tout idéal $I$ est complet. Le quotient $T_n(\V)/ I$ de $T_n(\V)$ est donc une $\V$-algèbre de Banach pour la topologie induite par le passage au quotient.
L'algèbre de Tate $T_n(\V)$ est l'ensemble des séries entières en $T$ à coefficients dans $\V$ qui convergent sur la boule unité fermée de $K^n$.
On peut aussi munir $T_n(\V)$ de la norme supérieure. Elle coïncide avec la norme de Gauss. Cela provient du principe du maximum vérifié par $T_n(\V)$ : il existe $y \in \V^n$ tel que $|f| = |f(y)|$.

\vspace{0.4cm}

Une $\V$-algèbre affinoïde $A$ est par définition une $\V$-algèbre de Banach isomorphe (en tant qu'algèbre topologique) à un quotient $T_n(\V)/I$ de $T_n(\V)$ par un idéal $I$.
Toutes les normes sur $A$ induites par une présentation de $A$ comme quotient d'une algèbre de Tate sont équivalentes.

Si $z$ est un idéal maximal de $A_K := A \otimes_\V K$, alors $A_K /z$ est une extension finie de $K$. La valeur absolue de $K$ s'étend uniquement en une valeur absolue sur $A_K /z$ notée encore $| \cdot |$.
On définit la norme spectrale d'un élément $f \in A_K$ de la manière suivante.
On note $f(z)$ l'image de $f$ dans $A_K/z$ et $|f(z)|$ sa valeur absolue. Alors
\[ \| f \|_\sp := \max_{z \in \Spm A_K} |f(z)|.\]
En général, $\| \cdot \|_\sp$ est seulement une semi-norme inférieure à toute norme de Gauss induite.
Cependant, lorsque l'algèbre $A_K$ est intègre, c'est une valeur absolue ultramétrique équivalente aux normes de Gauss. C'est le cas par exemple pour $A = T_n(\V)$.

\vspace{0.4cm}

Tout ouvert affine $U$ de $\X$ est le spectre formel d'une $\V$-algèbre affinoïde $A$: $U = \Spf A$.
De plus, $U_K = \Spm A_K$, où $A_K = A \otimes_\V K$ une $K$-algèbre affinoïde (ie un quotient de $T_n(K)$).
Puisque la courbe $\X$ est connexe et lisse, l'algèbre $A_K$ est intègre.
La norme spectrale $\| \cdot \|_\sp$ est donc une valuation complète sur l'algèbre affinoïde définissant $U_K$.

\vspace{0,4cm}

On suppose pour la fin de cette partie que $x$ est un point $\kappa$-rationnel de $X$.
Pour tout $0 \leq \lambda < 1$, on note $V_\lambda := \{ y \in U_K : |t(y)| \geq \lambda \}$.
C'est un ouvert de $\X_K$ contenu dans $U_K$. Puisque l'ouvert $U$ est affine, l'ouvert $V_\lambda$ est affinoide et ne dépend pas du choix de $t$ pour tout $\lambda$ vérifiant $| \lambda |  > |\omega| = \frac{1}{p}$.
Puisque la courbe $\X$ est lisse en $x$, on dispose d'un isomorphisme permettant d'identifier le tube $]x[$ à un disque ouvert :
\[ ]x[ \iso D(0, 1^-) := \{ y \in \A : 0 \leq | t(y) | < 1 \}. \]

Soit $f \in \Gamma( V_{\lambda_0} , \O_{\X_K})$ une section de $\O_{\X_K}$. Alors $f_{| ]x[\cap V_{\lambda_0}}$ s'écrit uniquement comme une série $\sum_{i\in \N} \alpha_i \cdot t^i$, où les $\alpha_i$ sont des éléments de $K$.
Cette fonction converge sur la couronne $C([\lambda_0 , 1[) := \{ y \in \A : \lambda_0 \leq | t(y) | < 1 \}$. Pour tout $\lambda_0 \leq \lambda < 1$, on note
\[ N(f_{| ]x[\cap V_\lambda} , \lambda) := \max\left\{i \in \N : |\alpha_i| \cdot \lambda^i = \sup_{j\in \N}{ |\alpha_j| \cdot \lambda^j} \right\} \in \N \cup \{+ \infty\}. \]
On pose
\[ N(f) := \lim_{\lambda \to 1^-} N(f_{| ]x[\cap V_\lambda} , \lambda) \in \N \cup \{+\infty \} .\]

\begin{lemma}\label{lemme1.1}
Pour toute section $f \in \Gamma( V_{\lambda_0} , \O_{\X_K})$ non nulle, $N(f)$ est un entier positif ne dépendant pas du choix de $t$.
De plus, si $f_{| ]x[\cap V_{\lambda_0}} = \sum_{i\in \N} \alpha_i \cdot t^i$, alors $N(f)$ est le plus petit indice tel que $ \|f\|_\sp = |\alpha_{N(f)}| = \max_{j\geq 0} |\alpha_j|$.
En particulier, $\|f\|_\sp$ est dans $|K|$.
\end{lemma}

\begin{remark}~
\begin{enumerate}
\item
Si $N(f)=0$, alors $f$ n'a pas de zéro sur $]x[$ et $x\in U_{\{f\}}$.
\item
On a $N(0, \lambda ) = N(0) = +\infty$.
\end{enumerate}
\end{remark}

On rappelle que $\O_{X,x}$ est un anneau de valuation discrète, de corps résiduel $\kappa$ lorsque $x$ est un point $\kappa$-rationnel.
Par définition, $t$ en est une uniformisante. On considère la valuation $v$ de $\O_{X,x}$ donnée par $v(t) = 1$.

\begin{lemma}\label{lemmevaluation}
Soit $f \in \Gamma(V_\lambda , \O_{\X_K})$ une section telle que $\| f \|_\sp = 1$. Alors $N(f)$ est la valuation de $(f \mod \varpi)$ dans $\O_{X,x}$.
\end{lemma}

On écrit $f_{| ]x[\cap V_{\lambda_0}} = \sum_{i\in \N} \alpha_i \cdot t^i$, toujours sous l'hypothèse que $\| f \|_\sp = 1$.
Autrement dit, les coefficients $\alpha_i$ sont dans $\V$.
On a $(f_{| ]x[\cap V_{\lambda_0}} \mod \omega) = \sum_{i \in \N} \bar{\alpha}_i \cdot t^i$ avec $\bar{\alpha}_i$ la réduction modulo $\varpi$ de $\alpha_i$.
Comme les coefficients $\alpha_i$ convergent pour la topologie $\varpi$-adique, les $\bar{\alpha}_i$ sont presque tous nuls et la somme définissant $(f_{| ]x[\cap V_{\lambda_0}} \mod \omega)$ est finie.
Alors $N(f)$ est le plus petit entier $n$ tel que $\bar{\alpha}_n \neq 0$.

\vspace{0.4cm}

Lorsque $U$ est un ouvert affine de $\X$, on note la norme spectrale de l'algèbre affinoïde $\O_{\X , \Q}(U) := \O_\X(U)\otimes_\V K$ simplement par $| \cdot|$.
On rappelle qu'elle est équivalente à toute norme de Gauss induite sur $\O_{\X , \Q}(U)$ et qu'il s'agit d'une valuation.

\subsection{Propriétés du faisceau $\Dkq$}\label{partie1.2}

\subsubsection*{Le faisceau $\Dkq$}

On commence par rappeler brièvement la définition du faisceau $\Dkq$ des opérateurs différentiels sur lequel on travaille.
Le lecteur peur regarder la seconde partie de l'article \cite{huyghe} de Christine Huyghe, Tobias Schmidt et Matthias Strauch pour plus de détails.
On désigne toujours par $U$ un ouvert affine contenant $x$ sur lequel on dispose d'une coordonnée étale.
On note $\partial$ la dérivation associée.

\vspace{0.4cm}

Le faisceau $\D^{(0)}_{\X , k}$ est défini comme un sous-faisceau dépendant d'un paramètre $k \in \N$ appelé niveau de congruence du faisceau usuel $\mathcal{D}^{(0)}_\X$ des opérateurs différentiels.
On retrouve $\mathcal{D}^{(0)}_\X$ lorsque $k=0$.
Localement, $\D^{(0)}_{\X , k}(U)$ est la $\V$-algèbre engendrée par $\O_\X(U)$ et par la dérivation $\varpi^k\partial$.
Plus précisément,
\[ \D^{(0)}_{\X , k}(U) = \left\{ \sum_{n \in \N}  a_n \cdot (\varpi^{k}\partial)^n , ~~ a_n \in \O_\X(U), ~~ a_n = 0 ~\mathrm{pour}~ n >>0 \right\}. \]
On peut aussi voir $\D^{(0)}_{U , k}$ comme le $\O_U$-module libre de base les puissances de $\varpi^{k} \partial$ :
\[ \D^{(0)}_{U , k} = \bigoplus_{n\in \N} \O_U \cdot (\varpi^k \partial)^n .\]

On note $\Dks$ la réduction modulo $\varpi$ du faisceau $ \D^{(0)}_{\X , k}$.
C'est le faisceau de $\kappa$-algèbres sur la fibre spéciale $X = \X \times_\V \Spec \kappa$ de $\X$ engendré localement sur $U$ par ${\O_X}_{|U}$ et par la dérivation $\partial_k$ image de $\varpi^k \partial$ après réduction modulo $\varpi$.
On rappelle que $\X$ et $X$ ont même espace topologique. On identifie donc $U$ à un ouvert affine de $X$.

\vspace{0.4cm}

Soit $\Dk := \varprojlim_i \left(\D^{(0)}_{\X , k}/ \varpi^{i+1} \D^{(0)}_{\X , k}\right)$ le complété $\varpi$-adique du faisceau $ \D^{(0)}_{\X , k}$ et $\Dkq := \Dk \otimes_\V K$. On dispose de la description locale suivante:
\[ \Dkq(U) = \left\{ \sum_{n=0}^\infty  a_n \cdot (\varpi^k\partial)^n , ~~ a_n \in \O_{\X , \Q} (U) , ~~| a_n |\underset{n \to \infty}{\longrightarrow} 0 \right\} .\]

Il est démontré dans \cite{huyghe} que ces algèbres sont toutes noetheriennes et que les faisceaux associés sont cohérents.
Pour $k'>k$, il est clair que $\widehat{\mathcal{D}}^{(0)}_{\X, k'}(U) \subset \Dk(U)$.
En particulier, $\Dkq(U)$ est une sous-algèbre de l'algèbre des opérateurs différentiels $\widehat{\mathcal{D}}^{(0)}_\X(U) = \widehat{\mathcal{D}}^{(0)}_{\X, 0}(U)$.
Cependant, lorsque $k \geq 1$, on observe que l'algèbre $\Dk(U)$ n'est pas isomorphe à l'algèbre $\widehat{\mathcal{D}}^{(0)}_\X(U)$.
On peut en effet le montrer en remarquant que le commutateur $[\varpi^k \partial , a ] = \varpi^k \cdot \partial(a)$ dans $\Dk(U)$ diffère du commutateur $[\partial , a ] = \partial(a)$ dans $\widehat{\mathcal{D}}^{(0)}_\X(U)$.

\subsubsection*{Structure d'algèbre de Banach sur $\Dkq(U)$}

On munit maintenant la $K$-algèbre $\Dkq(U)$ d'une norme multiplicative complète $\| \cdot \|_k$. Dans un premier temps, on suppose encore que $x$ est un point $\kappa$-rationnel de $X$.

\begin{definition}
Soit $H = \sum_{n \in \N} a_n \cdot (\varpi^k \partial)^n$ un élément de $\Dkq(U)$. On pose
\begin{enumerate}
\item
$\|H\|_k := \max_{n \geq 0} \{ |a_n| \}$ ;
\item
$\Nb(H) := \max\{ n \in \N : |a_n| = \|H\|_k \}$ ;
\item
$N_k(H) := N(a_{\Nb(H)})$.
\end{enumerate}
\end{definition}

On rappelle que si $\sum_{i \in \N} \alpha_i \cdot t^i$ est l'écriture comme série de $a_{\Nb(H)}$ sur $]x[ \cap U_K$, alors $\|H\|_k = |\alpha_{N_k(H)} |$.
On peut remarquer que
\[ \Dk(U) =  \left\{ H \in \Dkq(U)  : \| H \|_k \leq 1 \right \}. \]

Soit $H = \sum_{n \geq 0} a_n \cdot (\varpi^k \partial)^n$ un opérateur différentiel non nul de $\Dkq(U)$.
On fixe un scalaire $\alpha \in K^\times$ tel que $| \alpha | = \left(\max_{n \geq 0} |a_n| \right)^{-1}$.
Il s'agit bien d'un élément de $ |K|^\times$ d'après le lemme \ref{lemme1.1}. Alors $\alpha H$  est de norme 1 et l'opérateur $\alpha H$ appartient à $\Dk(U)$.
L'entier $\Nb(H)$ est le plus grand indice $n$ tel que $| \alpha  \cdot a_n | = \|\alpha H\|_k= 1$.
Ainsi, le nombre $\Nb(H)$ est l'ordre de l'opérateur $(\alpha H \mod \varpi)$ dans la $\kappa$-algèbre $\Dks(U)$. Cet entier ne dépend pas du choix de $\alpha$.
De plus, $N_k(H) = N_k(\alpha H)$ est la valuation de $\alpha \cdot a_{\Nb(H)}$ modulo $\varpi$ dans $\O_{X,x}$ d'après le lemme \ref{lemmevaluation}.
Ce nombre ne dépend pas non plus de $\alpha$.

\vspace{0.4cm}

Les entiers $\Nb(H)$ et $N_k(H)$ coïncident donc respectivement avec l'ordre et la valuation de $(\alpha H \mod \varpi)$ dans $\Dks(U)$ pour tout scalaire $\alpha \in K$ vérifiant $\|\alpha H \|_k = 1$.
Par ailleurs, ces définitions sont indépendantes du choix de la coordonnée locale sur $U$.

\begin{lemma}
La norme $\| \cdot \|_k$ et les fonctions $\Nb$ et $N_k$ ne dépendent pas du choix de la coordonnée locale.
\end{lemma}
\begin{proof}
On considère une autre coordonnée locale sur l'ouvert $U$ ; on note $\partial'$ la dérivation associée.
Puisque $\partial'$ est un générateur du faisceau tangent $\O_U \cdot \partial$, il existe un élément $\alpha$ de $\O_\X(U)^\times$ tel que $\partial' = \alpha \cdot \partial$.
Comme  $|\alpha| \leq 1$ et $\alpha$ est inversible, on a $|\alpha| = 1$.
Soit $P = \sum_{n \in \N} a_n \cdot (\varpi^k \partial')^n $ un opérateur différentiel de $\Dkq(U)$.
Sa norme $\|P\|_k$ pour la dérivation $\partial'$ est le maximum des normes spectrales des coefficients $a_n$.

Par ailleurs, $P = \sum_{n \in \N} a_n \cdot (\alpha\varpi^k \partial)^n$. On a $(\alpha\partial)^2 = \alpha^2 \partial^2 + \alpha \partial(\alpha) \partial$.
Comme $|\partial^n(\alpha)| \leq |\alpha| = 1$, le coefficient de $\partial$ a une norme spectrale inférieure ou égale à un.
Une récurrence sur $n \geq 1$ montre que
\[ (\alpha\partial)^n = \alpha^n \partial^n + \sum_{m=0}^{n-1} b_m \partial^m \]
avec $|b_m| \leq 1$ pour tout entier $m \in \{ 0 , \dots , n-1 \}$.
Il vient
\begin{align*}
P & = \sum_{n \in \N} a_n \left[ \alpha^n (\varpi^k\partial)^n + \varpi^{kn} \sum_{m=0}^{n-1} b_m \partial^m\right] \\
& = \sum_{n \in \N} a_n \alpha^n (\varpi^k\partial)^n + \underbrace{\sum_{n\in \N^*} a_n \varpi^k \sum_{m=0}^{n-1} \varpi^{k(n-m-1)} b_m (\varpi^k\partial)^m}_{\sum_{n\geq 0} \beta_n (\varpi^k \partial)^n}
\end{align*}
avec $|\beta_n| \leq |\varpi|^k \cdot \|P\|_k $ et $|a_n \alpha^n | = |a_n|$. Lorsque $k>0$, $|\beta_n| < \|P\|_k$ ; il est clair que la norme de $P$ pour la dérivation $\partial$ est aussi donnée par le maximum $\max_{n\in\N} |a_n|$.
Pour $k=0$, le résultat reste vrai. En effet, dans la seconde somme, le coefficient de $(\varpi^k \partial)^n$ est une combinaison des $a_k$ pour $k>n$ et des $b_m$.
\end{proof}

On rappelle que la norme $\|P\|_k$ d'un opérateur différentiel $P$ de $\Dkq(U)$ et que les entiers $\Nb(P)$ et $N_k(P)$ dépendent du point $x$ donné.
Ce sont des notions locales en $x$.
On rappelle aussi, lorsque $\|P\|_k  = 1$, que $N_k(P)$ est la valuation du coefficient dominant de $\bar{P} = (P \mod \varpi)$ dans $\O_{X,x}$.

\begin{prop}~
\begin{enumerate}
\item
Les algèbres $\Dk(U)$ et $\Dkq(U)$ sont complètes pour la norme $\| \cdot \|_k$.
\item
La topologie induite par quotient sur tout $\Dkq$-module cohérent est complète.
\item
Pour tout opérateurs $H$ et $Q$ de $\Dkq(U)$, on a
\[ \| HQ \|_k = \| H\|_k \cdot \|Q\|_k,  ~~ \Nb(HQ) = \Nb(H) + \Nb(Q),  ~~ N_k(HQ) = N_k(H) + N_k(Q). \]
\end{enumerate}
\end{prop}

\begin{proof}
Le premier point découle du fait que l'algèbre $\Dkq(U)$ est complète pour la topologie $\varpi$-adique et que la topologie induite par la norme spectrale est équivalente à la topologique $\varpi$-adique sur l'algèbre affinoïde $\O_{\X,\Q}(U)$.
On munit tout $\Dkq$-module cohérent $\E$ de la norme induite par des présentations finies locales de $\E$.
Cette dernière est complète et ne dépend pas des présentations choisies puisque la norme $\| \cdot \|_k$ est multiplicative par 3.
Cela montre le second point.

Soit maintenant $H = \sum_{n \in \N} a_n \cdot (\varpi^k \partial)^n$ et $Q = \sum_{n \in \N} b_n \cdot (\varpi^k \partial)^n$ deux opérateurs différentiels de $\Dkq(U)$. On a
\begin{align*}
H Q & = \sum_{i\in \N} a_i \cdot (\varpi^k \partial)^i \left(\sum_{j \in \N} b_j \cdot (\varpi^k \partial)^j \right) \\
& = \sum_{i , j \geq 0} \left( \sum_{\ell = 0}^i {i \choose \ell} \cdot a_i \cdot \varpi^{k\ell} \cdot  \partial^\ell(b_j) \cdot \varpi^{k(i+j-\ell)} \cdot\partial^{i+j-\ell} \right) \\
& = \sum_{u \geq 0} \underbrace{\underset{0 \leq j \leq u}{\sum_{\ell \geq 0}} \left( {u+\ell-j \choose \ell} \cdot a_{u+\ell-j} \cdot \varpi^{k\ell} \cdot \partial^\ell(b_j) \right)}_{\alpha_u \in \O_{\X , \Q}(U)} (\varpi^k \partial)^u .
\end{align*}
On remarque déjà que
\[ \left| {u+\ell-j \choose \ell} \cdot a_{u+\ell-j} \cdot \varpi^{k\ell} \cdot \partial^\ell(b_j) \right| \leq |a_{u+\ell-j}| \cdot | \partial^\ell(b_j)| \leq |a_{u+\ell-j}| \cdot |b_j| \leq \| H \|_k \cdot \| Q \|_k .\]
Ainsi, $\|HQ\|_k \leq \| H \|_k \cdot \| Q \|_k$. Pour $u = \Nb(H) + \Nb(Q)$, $\ell = 0$ et $ j = \Nb(Q)$, le coefficient associé dans la somme définissant $\alpha_u$ est $a_{\Nb(H)} \cdot b_{\Nb(Q)}$.
Ce terme est donc de norme $\| H \|_k \cdot \| Q \|_k$. Si $j \geq \Nb(Q)$, alors $|b_j| < \| Q \|_k$.
Si $j < \Nb(Q)$ ou si $j \leq \Nb(Q)$ et $\ell \geq 1$, alors $u+\ell-j > \Nb(H)$. Donc $|a_{u+\ell-j}| < \| H \|_k$.
Dans tous ces cas, la norme du terme associé dans $\alpha_u$ est strictement inférieure à $\| H \|_k \cdot \| Q \|_k$.
Ceci prouve que $|\alpha_u| = \| H \|_k \cdot \| Q \|_k$. Autrement dit, $\|HQ\|_k = \| H \|_k \cdot \| Q \|_k$.

Si $u > \Nb(H) + \Nb(Q)$, on montre de manière analogue que $|\alpha_u| < \| H \|_k \cdot \| Q \|_k$.
Ainsi, $\Nb(HQ) = \Nb(H) + \Nb(Q)$. On peut supposer que $H$ et $Q$ sont de norme un.
Dans ce cas, $N_k(H) = v(a_{\Nb(H)} \mod \varpi)$ et $N_k(Q) = v(b_{\Nb(Q)} \mod \varpi)$, où $v$ est la valuation de $\O_{X,x}$. Puisque $\alpha_{\Nb(H) + \Nb(Q)} = a_{\Nb(H)} \times b_{\Nb(Q)} +$ (un terme de norme spectrale strictement inférieure), on a bien
\begin{align*}
N_k(HQ) & = v(a_{\Nb(H)} \cdot b_{\Nb(Q)} \mod \varpi) = v(a_{\Nb(H)} \mod \varpi) + v(b_{\Nb(Q)} \mod \varpi) \\
& = N_k(H) + N_k(Q) .
\end{align*}
\end{proof}

\subsubsection*{Applications}

On énonce dans cette partie quelques propriétés de l'algèbre de Banach $\Dkq(U)$.
Les preuves sont adaptées de celles de Laurent Garnier à un niveau de congruence $k$ quelconque.
La proposition suivante caractérise l'inversibilité des éléments de $\Dkq$ à l'aide des fonctions $\Nb$ et $N_k$. 

\begin{prop}\label{prop1.7}
On suppose que $x$ est un point $\kappa$-rationnel. Soit $H \in \Dkq(U)$. Il existe un ouvert $V$ de $U$ contenant $x$ sur lequel $H$ est inversible si et seulement si $\Nb(H) = N_k(H) = 0$. Si de plus $\|H\|_k = 1$, alors $H^{-1} \in \Dk(V)$.
\end{prop}
\begin{proof}
Si $H$ est inversible d'inverse $H^{-1}$ , alors $\Nb(H) + \Nb(H^{-1}) = \Nb(1) = 0$. Donc $\Nb(H) = 0$ puisque $\Nb(H)$ est un entier positif.
De même, $N_k (H) = 0$. Réciproquement, on suppose que $\Nb(H) = N_k(H) = 0$.
On écrit $H = \sum_{n \in \N} a_n \cdot (\varpi^k \partial)^n$. Ces deux conditions signifient que $|a_0| > |a_n|$ pour tout $n>0$ et que $a_0$ n'a pas de zéro sur $]x[$.
Autrement dit, $a_0$ est inversible sur l'ouvert $V = U_{\{a_0\}} \cup \{x\}$ de $U$. Sur cet ouvert, l'inverse de $H$ est donné par la série classique
\[ H^{-1} = \sum_{i \in \N} \left(-\sum_{j\in \N} \frac{a_j}{a_0} (\varpi^k\partial)^j \right)^i a_0^{-1} .\]
Cet opérateur converge puisque
\[ \left\| \sum_{j\in \N} \frac{a_j}{a_0} (\varpi\partial)^j \right\|_k = \max_{j \geq 1} \left\{ \left| \frac{a_j}{a_0} \right| \right\} <1 .\]
Ainsi, $H^{-1}$ définit bien un opérateur de $\Dkq(V)$. Si maintenant $H$ est de norme un, alors les coefficients $a_n$ et $a_0^{-1}$ sont des éléments de $\O_\X(V)$. Il en découle que $H^{-1} \in \Dk(V)$.
\end{proof}

On fixe une clôture algébrique $\overline{K}$ de $K$. A partir de maintenant, et pour le reste de l'article, $x$ n'est plus supposé $\kappa$-rationnel. C'est un point $\kappa'$-rationnel pour une certaine extension finie $\kappa'$ de $\kappa$.
Soit $K'$ une extension finie de $K$ dans $\overline{K}$ dont le corps résiduel est $\kappa'$. Quitte à étendre $K$ par $K'$, on peut définir les fonctions $\Nb$ et $N_k$ des opérateurs de $\Dkq(U)$ en $x$.

\vspace{0.4cm}

Puisque l'extension $K' /K$ est finie, l'algèbre $K\langle T_1 , \dots , T_n\rangle \otimes_K K'$ est complète.
La $K'$-algèbre de Tate $T_n(K')$ coïncide donc avec $T_n(K)\otimes_K K'$. On munit $K'$ de l'extension non normalisée de la valeur absolue de $K$, notée encore $| \cdot |$.
Le morphisme canonique $T_n(K) \to T_n(K')$ est une isométrie de $K$-algèbres pour les normes de Gauss, égales aux normes spectrales.
Plus généralement, si $A$ est une $K$-algèbre affinoïde, alors $A' = A \otimes_K K'$ est une $K'$-algèbre affinoïde.
Le morphisme canonique $A \to A'$ est une isométrie de $K$-algèbres affinoïdes. Lorsque $A$ est intègre, la norme spectrale est une norme sur $A$ et le morphisme précédent est une isométrie pour les normes spectrales. 

\vspace{0.4cm}

On munit $\Dkq(U) \otimes_K K'$ de la norme de $K'$-algèbre $\| P \otimes \lambda\|'_k = |\lambda| \cdot \| P\|_k$.
Comme le morphisme canonique $\O_{\X  , \Q}(U) \to \O_{\X  , \Q}(U) \otimes_K K'$ est une $K$-isométrie, le morphisme $ \Dkq(U) \to \Dkq(U) \otimes_K K'$ est une isométrie de $K$-algèbres. Soit $H \in \Dkq(U)$.
La fonction $\Nb(H)$ ne dépend donc pas de l'extension $K'$ de $K$ mais seulement de $H$ : cet entier est le même aussi bien dans $(\Dkq(U) , \| \cdot \|_k)$ que dans $(\Dkq(U) \otimes_K K' , \| \cdot \|_k')$.

\begin{cor}\label{corinv}
Un opérateur différentiel $H$ de $\Dkq(U)$ est inversible au voisinage de $x$ si et seulement si $\Nb(H) = N_k(H) = 0$.
\end{cor}
\begin{proof}
La proposition \ref{prop1.7} montre que $H$ est inversible au voisinage de $x$ après extension des scalaires de $K$ à $K'$.
Soit $V \subset U$ un ouvert contenant $x$ sur lequel $H$ est inversible. On écrit $H = \sum_{n=0}^\infty a_n \cdot (\varpi^k \partial)^n \in \Dkq(V)$.
Puisque $a_0$ est inversible dans $\O_{\X , \Q}(V) \otimes_K K'$ et $a_0 \in \O_{\X , \Q}(V)$, $a_0$ est inversible dans $\O_{\X ,  \Q}(V)$.
L'inverse $H^{-1} = \sum_{i \geq 0} \left(-\sum_{j\geq1} \frac{a_j}{a_0} (\varpi^k\partial)^j \right)^i a_0^{-1}$ de $H$ dans $\Dkq(V) \otimes_K K'$ appartient donc à $\Dkq(V)$.
\end{proof}

Ce critère d'inversibilité permet de démontrer que la $K$-algèbre $\Dkq(U)$ est simple.

\begin{prop}\label{prop1.10}
Pour tout ouvert affine $V$ de $\X$, $\Dkq(V)$ est une algèbre simple.
\end{prop}
\begin{proof}
Soit $I$ un idéal bilatère non nul de $\Dkq(V)$ et $x \in V$ un point fermé.
On va montrer qu'il existe un voisinage ouvert affine $W$ de $x$ dans $V$ tel que $I_{|W}$ contienne un élément inversible dans $\Dkq(W)$.
Les points fermés étant denses dans $V$, ceci implique que $I = \Dkq(V)$.
D'après le corollaire \ref{corinv}, il suffit de montrer que quitte à réduire $V$, $I$ contient un élément $P$ vérifiant $\Nb(P) = N_k(P) = 0$. On peut remplacer $K$ par une extension finie afin que $x$ soit rationnel et supposer que $V$ est affine.

On part d'un opérateur différentiel non nul $H = \sum_{i \in \N} a_i \cdot (\varpi^k \partial)^i$ de $I$.
Comme $I$ est un idéal bilatère, les crochets $[ H , t] = Ht - tH$ et $[H , t]^{n+1} := [[H , t]^n , t]$ pour $n \in \N$ restent des éléments de $I$.
On a $[H , t ] = \varpi^k \cdot  \sum_{i \in \N^*} i a_i \cdot (\varpi^k \partial)^{i-1}$ et 
\[ [H , t]^{\Nb(H)} = (\varpi^{k \Nb(H)} \cdot \Nb(H) !)  \sum_{i \geq \Nb(H)} {i \choose \Nb(H)} \cdot a_i \cdot (\varpi^k \partial)^{i-\Nb(H)} .\]
Pour tout $i > \Nb(H)$, on a
\[ \left| {i \choose \Nb(H)} a_i \right| \le | a_i | < | a_{\Nb(H)} | .\]
Autrement dit, $\Nb( [H , t]^{\Nb(H)} ) = 0$. Quitte à remplacer $H$ par $[H , t]^{\Nb(H)}$, on peut supposer que $\Nb(H) = 0$.
Par ailleurs, $\varpi^k \cdot \partial \cdot a_i \cdot (\varpi^k \partial)^i = \varpi^k \cdot \partial(a_i) \cdot (\varpi^k \partial)^i + a_i \cdot (\varpi^k \cdot  \partial)^{i+1}$.
Donc
\begin{align*}
[H , \varpi^k\partial ] & = H \varpi^k\partial - \varpi^k \partial H = \sum_{i \geq 0} \left(a_i \cdot (\varpi^k \partial)^{i+1} - \varpi^k \partial \cdot a_i \cdot (\varpi^k \partial)^i \right) \\
& = -\varpi^k \sum_{i \geq 0} \partial(a_i) \cdot (\varpi^k \partial)^i.
\end{align*}
Ainsi, $[H , \varpi^k\partial ]^{N_k(H)} =  (-\varpi^k)^{N_k(H)} \sum_{i \geq 0} \partial^{N_k(H)}(a_i) \cdot (\varpi^k \partial)^i$. Puisque $\Nb(H) = 0$, on a
\[ \forall i \geq 1, ~~  | \partial^{N_k(H)}(a_i) | \leq | N_k(H)!| \cdot | a_i |  < |N_k(H)!| \cdot \|H\|_k. \]
Sur $]x[ \cap U_K$ on peut écrire $a_0 = \sum_{i\geq 0} \alpha_i \cdot t^i$, $\alpha_i \in K$. On a
\[ \partial^{N_k(H)} (a_0) = N_k(H)! \sum_{i\geq N_k(H)} {i \choose N_k(H)} \cdot \alpha_i \cdot t^{i-N_k(H)} .\]
Comme $N_k(H) = N(a_0)$, on a
\[ \forall i > N_k(H), ~~ \left| {i \choose N_k(H)} \alpha_i \right| \leq | \alpha_i | < | \alpha_{N_k(H)} | = |\alpha_{N(a_0)} | = \| H \|_k .\]
Ainsi, $| \partial^{N_k(H)}(a_0) | = |N_k(H)!| \cdot |\alpha_0| = |N_k(H)!| \cdot \|H\|_k$ et $N_k(\partial^{N_k(H)}(a_0)) = 0$. Ceci montre que $[H , \varpi^k\partial]^{N_k(H)}$ est un élément de $I$ de fonctions $\Nb$ et $N_k$ nulles. Quitte à réduire l'ouvert $V$ contenant $x$, cet élément est inversible.
\end{proof}

\subsection{Théorèmes de division dans $\Dkq$}\label{partie1.3}

Les résultats de cette partie sont une adaptation des théorèmes de division énoncés par Laurent Garnier dans \cite{garnier} pour $\widehat{\mathcal{D}}^{(0)}_{\X, \Q}$ au cas des $\Dkq$-modules cohérents.
Les preuves se généralisent immédiatement pour un niveau de congruence $k$.

\begin{definition}
Soit $P$ un opérateur différentiel de $\Dkq(U)$.
\begin{enumerate}
\item
On appelle coefficient dominant de $P$ son coefficient d'indice $\Nb(P)$.
Si $\|P \|_k =1$, il s'agit du coefficient dominant de $\bar{P}$ après réduction modulo $\varpi$.
\item
On dit que $P$ est $\Nb$-dominant si $P$ est un opérateur fini d'ordre $\Nb(P)$.
Cette condition signifie que le coefficient de plus haut degré de $P$ est de norme maximale, ou de manière équivalente que $P$ et $\bar{P}$ ont le même ordre lorsque $\| P \|_k =1$.
\end{enumerate}
\end{definition}

\begin{prop}
Soit $P$ un opérateur différentiel non nul de $\Dkq(U)$. On note $b$ son coefficient dominant et $V$ l'ouvert $U_{\{b\}} \cup \{x\} $ de $U$.
Alors tout élément $H$ de $\Dkq(U)$ s'écrit uniquement sous la forme $H = Q P + R + S$ avec:
\begin{enumerate}
\item
$Q, R, S \in \Dkq(V)$ ;
\item
$R$ est d'ordre fini $< \Nb(P)$ ;
\item
$ S = \sum_{i \geq \Nb(P)} \mu_i \cdot (\varpi^k \partial) ^i $, $\mu_ i\in K[t]$ de degré $< N_k(P)$ ;
\item
$\| H \|_k = \max\{ \| Q \|_k \cdot \| P \|_k, \| R \|_k , \| S \|_k \}$.
\end{enumerate}
Si $H \in \Dk(U)$, alors $R$ et $S$ sont dans $\Dk(V)$. Si de plus $\| P \|_k = 1$, alors $Q \in \Dk(V)$.
\end{prop}

Si $N_k(P) = 0$, alors $S = 0$ puisque ses coefficients sont des polynômes de degrés strictement inférieurs à $N_k(P)$.
En se restreignant à l'ouvert $V = U_{\{ b \}}$, on peut factoriser $P$ par $b$ et supposer que $N_k(P) = 0$.
On en déduit le corollaire suivant.

\begin{cor}
Soit $P$ un opérateur différentiel non nul de $\Dkq(U)$ de coefficient dominant $b$. Si $V = U_{\{b\}}$, alors tout élément $H$ de $\Dkq(U)$ s'écrit uniquement sous la forme $H = Q P + R$ avec:
\begin{enumerate}
\item
$Q, R \in \Dkq(V)$ ;
\item
$R$ est d'ordre fini $< \Nb(P)$ ;
\item
$\| H \|_k = \max\{ \| Q \|_k \cdot \| P \|_k, \| R \|_k \}$.
\end{enumerate}
\end{cor}

Ces théorèmes de divisions permettent de démontrer deux versions du lemme de Hensel pour tout opérateur différentiel $P$ de $\Dkq(U)$.

\begin{prop}[Lemme de Hensel]
Soit $H$ un opérateur non nul de $\Dkq(U)$ de coefficient dominant $b$.
On note encore $V = U_{\{b\}} \cup \{x\}$. Alors $H$ se décompose uniquement sous la forme $H = QP + S$ avec
\begin{enumerate}
\item
$Q, P, S  \in \Dkq(V)$ ;
\item
$P$ est $\Nb$-dominant de coefficient dominant $b$ ;
\item
$S = \sum_{i \geq \Nb(H)} \mu_i \cdot (\varpi^k \partial)^i$ avec $\mu_i \in K[t]$ de degré $< N_k(P)$ ;
\item
$\| Q \|_k = 1$ et il existe un ouvert $W \subset U$ tel que $Q$ soit inversible dans $\Dkq(W)$ ;
\item
$\| S \|_k < \| H \|_k$.
\end{enumerate}
\end{prop}

En ne cherchant plus à énoncer une division sur un ouvert contenant $x$, on obtient la version suivante du lemme d'Hensel.

\begin{prop}[Lemme de Hensel]\label{lemmehensel}
Soit $H \in \Dkq(U) \backslash\{0\}$ de coefficient dominant $b$.
Alors $H$ se décompose uniquement sous la forme $H = QP$ avec
\begin{enumerate}
\item
$Q, P   \in \Dkq(U_{\{b\}})$ ;
\item
$P$ est $\Nb$-dominant de coefficient dominant $b$ ;
\item
$\| Q \|_k = 1$ et $Q$ est inversible dans $\Dkq(U_{\{b\}})$.
\end{enumerate}
\end{prop}

Les deux corollaires suivant se déduisent de la division selon un opérateur différentiel de $\Dkq$.

\begin{cor}\label{corhensel}
Soit $\E = \Dkq / P$ un $\Dkq$-module cohérent à gauche donné par un opérateur différentiel $P$ de $\Dkq(U)$.
Il existe un ouvert $V$ de $U$ (obtenu en retirant les zéros du coefficients dominant de $P$) sur lequel $\E_{|V} \simeq \Dkq / \tilde{P}$ avec $\tilde{P}$ un opérateur $\Nb$-dominant de même coefficient dominant que $P$.
De plus, $\E_{|V}$ est un $\O_{\X,Q}$-module libre de rang $\Nb(P)$.
\end{cor}
\begin{proof}
On applique le lemme d'Hensel à $P$ avec $V$ l'ouvert sur lequel le coefficient dominant de $P$ est inversible.
On peut écrire $P = Q \tilde{P}$ avec $\tilde{P}$ vérifiant les conditions de l'énoncé et $Q$ un opérateur inversible dans $\Dkq(V)$.
On en déduit que $\E_{|V} \simeq \Dkq / \tilde{P}$.
La seconde partie de l'énoncé découle du théorème de division dans $\Dkq(V)$ puisque le coefficient dominant de $\tilde{P}$ est inversible sur l'ouvert $V$ : tout élément $H$ de $\Dkq(V)$ s'écrit uniquement sous la forme $H = Q \tilde{P} +R$ avec $R$ un opérateur fini de $\Dkq(V)$ d'ordre strictement inférieur à $\Nb(P)$.
\end{proof}

\begin{cor}\label{cor1.2.12}
Soient $P, Q \in \Dkq(U)$ tels que $\ \Dkq / P \simeq \Dkq / Q$ en tant que $\Dkq$-modules à gauche. Alors $\Nb(P) = \Nb(Q)$.
\end{cor}
\begin{proof}
Soit $V$ un ouvert contenu dans $U$ sur lequel $\widehat{\mathcal{D}}^{(0)}_{V, k , \Q} / P \simeq \widehat{\mathcal{D}}^{(0)}_{V, k , \Q} / \tilde{P}$ et $\widehat{\mathcal{D}}^{(0)}_{V, k , \Q} / Q \simeq \widehat{\mathcal{D}}^{(0)}_{V, k , \Q} / \tilde{Q}$ avec $\tilde{P}$ et $\tilde{Q}$ deux opérateurs finis d'ordre respectif $\Nb(P)$ et $\Nb(Q)$.
Ces deux modules sont des $\O_{V , \Q}$-modules libres de rang $\Nb(P)$ et $\Nb(Q)$ respectivement.
Puisqu'ils sont isomorphes en tant que $\widehat{\mathcal{D}}^{(0)}_{V, k , \Q}$-modules, ils sont isomorphes en tant que $\O_{V , \Q}$-modules.
On en déduit que $\Nb(P) = \Nb(Q)$.
\end{proof}

La proposition suivante provient de l'existence d'une division \og euclidienne \fg \, sur $\Dkq$ et du lemme d'Hensel (proposition \ref{lemmehensel}).
La preuve est analogue à celle de la proposition 5.1.2 de l'article \cite{garnier} de Laurent Garnier en rajoutant un niveau de congruence $k$.

\begin{prop}\label{propdec}
Soit $\E$ un $\Dkq$-module cohérent et $U$ un ouvert affine de $\X$ contenant $x$.
Il existe alors un opérateur $P$ de $\Dkq(U)$, un ouvert affine $V$ contenu dans $U$ (obtenu en retirant les zéros du coefficients dominant de $P$) et un entier $n$ tels que
\[ \E_{|V} \simeq  (\widehat{\mathcal{D}}^{(0)}_{V, k, \Q} / P) \oplus (\widehat{\mathcal{D}}^{(0)}_{V, k, \Q})^n .\]
\end{prop}

\subsection{Base de division d'un idéal cohérent de $\Dkq$}\label{partie1.4}

On termine cette section par définir une base de division d'un idéal cohérent non nul $\I$ de $\Dkq$.
Une telle base permettra de calculer la variété caractéristique du $\Dkq$-module cohérent $\Dkq / \I$. 

\vspace{0.4cm}

On commence par définir la notion de base de division en $x$ au niveau de la fibre spéciale $X$ de $\X$.
Soit $U$ un ouvert affine contenant $x$ admettant une coordonnée locale associée à $x$.
On note $\Dx := \Dks(U) \otimes_\kappa \O_{X,x}$. En tant que $\kappa$-algèbre, $\Dx$ est isomorphe à $\bigoplus_{n\in \N} \O_{X,x} \cdot \partial_k^n$ où $\partial_k$ est l'image de $\varpi^k \partial$ après réduction modulo $\omega$.
Il s'agit de l'algèbre des opérateurs différentiels en $\partial_k$ à coefficients dans $\O_{X,x}$. On note dans la suite l'algèbre $\Dks(U) \otimes_\kappa \O_{X,x}$ simplement par $\Dks \otimes_\kappa \O_{X,x}$ puisqu'elle ne dépend pas du choix de $U$.

On rappelle que $\O_{X,x}$ est un anneau de valuation discrète $v$ d'uniformisante $t$.
Les notions de base de division en $x$ vont coïncider entre un idéal cohérent de $\Dk$ et sa réduction modulo $\varpi$ dans $\Dx$.

\vspace{0.4cm}

Soit $P = \alpha_d \cdot \partial_k^d + \dots +\alpha_1 \cdot \partial_k+\alpha_0$ un opérateur non nul d'ordre $d=d(P)$ de $\Dx$. On appelle valuation de $P$ celle de son coefficient dominant $a_d$ : $v(P) := v(\alpha_d)$. \textit{L'exposant} $\Exp(P)$ de $P$ est le couple $(v(P) , d(P)) \in \N^2$. Si $Q$ est un autre opérateur de $\Dx$, on vérifie que $\Exp(PQ) = \Exp(P) + \Exp(Q)$.

\vspace{0.4cm}

Soit $I$ un idéal cohérent à gauche non nul de $\Dx$. On définit son \textit{exposant} par
\[ \Exp(I) := \{(v(P) , d(P)) : P \in I \backslash\{0\} \} \subset \N^2 .\]
On a
\[ \Exp(t^i \cdot P) =  ( i , 0)  + \Exp(P) ~~~ \mathrm{et} ~~~ \Exp(\partial_k^j \cdot P) = (0 , j) + \Exp(P) .\]
On en déduit que $\Exp(I) = \Exp(I) + \N^2$. Ainsi, l'exposant de $I$ est une partie de $\N^2$ délimitée inférieurement par un escalier fini.
On peut voir la figure ci-dessous pour un exemple.

Soit $P_1$ un élément de $I$ de degré minimal $d$ et de valuation minimale parmi les éléments de $I$ de degré $d$.
On construit récursivement un élément $P_i$ de $I$ d'ordre $d(P_{i+1}) = d(P_i) + 1$ et de valuation minimale parmi les éléments de même degré jusqu'à obtenir un élément $P_r$ de valuation minimale dans $I$.

On obtient ainsi une famille d'opérateurs $(P_1 , \dots, P_r)$ échelonnée pour l'ordre telle que $P_i$ soit de valuation minimale parmi les éléments de même ordre, telle que $d(I) = d(P_1)$ soit l'ordre minimal des éléments de $I$ et telle que $v(I) = v(P_r)$ soit la valuation minimale des éléments de $I$. Une telle famille est appelée \textit{base de division} de $I$.

\vspace{0.4cm}

Soit $I$ un idéal cohérent de $\Dks$ et $x \in X$. Alors $I_x$ est un idéal cohérent de $\Dx = \Dks \otimes_\kappa \O_{X,x}$.
On appelle \textit{base de division} de $I$ relativement au point $x$ une base de division $(P_1 , \dots , P_r)$ de l'idéal $I_x$.
Les opérateurs $P_1 , \dots , P_r$ sont des éléments de $I(U)$ pour un certain ouvert affine $U$ contenant $x$.

\vspace{0.4cm}

La figure ci-dessous illustre graphiquement le positionnement d'une base de division en $x$ vis-à-vis de l'exposant de $I$.

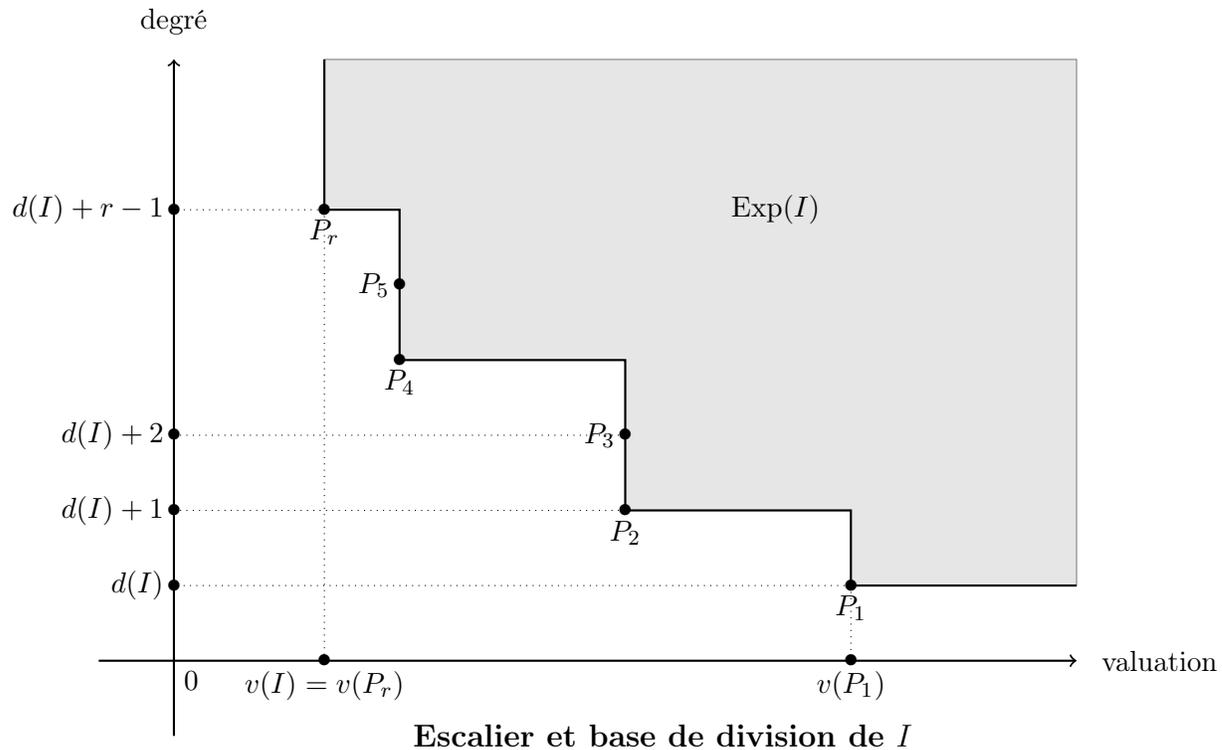
\begin{center}
\begin{tikzpicture}
\draw[thick][->] (-1,0) -- (12,0);
\draw (12.2,0) node[right] {valuation};
\draw [thick][->] (0,-1) -- (0,8);
\draw (0,8.2) node[above] {degré};
\draw (0,0) node[below right] {$0$};

\draw[gray][fill= gray!20]  (2,8) -- (2,6)  -- (3,6) -- (3,4) -- (6,4) -- (6 , 2) -- (9, 2) -- (9 , 1) -- (12 , 1) -- (12 , 8) -- (2,8) -- cycle ;

\draw[thick] (2,8) -- (2,6)  -- (3,6) -- (3,4) -- (6,4) -- (6 , 2) -- (9, 2) -- (9 , 1) -- (12 , 1);
\draw (9,1) node{$\bullet$} ;
\draw (9,1) node[below]{$P_1$} ;
\draw (6,2) node{$\bullet$} ;
\draw (6,2) node[below]{$P_2$} ;
\draw (6,3) node{$\bullet$} ;
\draw (6,3) node[left]{$P_3$} ;
\draw (3,4) node{$\bullet$} ;
\draw (3,4) node[below]{$P_4$} ;
\draw (3,5) node{$\bullet$} ;
\draw (3,5) node[left]{$P_5$} ;
\draw (2,6) node{$\bullet$} ;
\draw (2,6) node[below]{$P_r$} ;

\draw[dotted] (2,6) -- (2,0) ;
\draw (2,0) node{$\bullet$} ;
\draw (2,0) node[below]{$v(I) = v(P_r)$} ;

\draw[dotted] (9,1) -- (9,0) ;
\draw (9,0) node{$\bullet$} ;
\draw (9,0) node[below]{$v(P_1)$} ;

\draw[dotted] (9,1) -- (0,1) ;
\draw (0,1) node{$\bullet$} ;
\draw (0,1) node[left]{$d(I)$} ;

\draw[dotted] (6,2) -- (0,2) ;
\draw (0,2) node{$\bullet$} ;
\draw (0,2) node[left]{$d(I)+1$} ;

\draw[dotted] (6,3) -- (0,3) ;
\draw (0,3) node{$\bullet$} ;
\draw (0,3) node[left]{$d(I)+2$} ;

\draw[dotted] (2,6) -- (0,6) ;
\draw (0,6) node{$\bullet$} ;
\draw (0,6) node[left]{$d(I)+r-1$} ;

\draw (8, 6) node{$\Exp(I)$} ;

\draw (5.5 , -1) node{\large\textbf{Escalier et base de division de $I$} en $x$} ;

\end{tikzpicture}
\end{center}

\vspace{0,2cm}

Soit maintenant $\I$ un idéal à gauche cohérent non nul de $\Dkq$ et $Q \in \I_x$ ; $Q$ est un opérateur de $\Dkq(U)$ pour un certain ouvert affine $U$ contenant $x$.
On lui associe le couple $(N_k(Q) , \Nb(Q))$ ne dépendant que de $x$ appelé \textit{exposant} de $Q$ en $x$.
Si $Q$ est de norme un, on rappelle que $N_k(Q)$ et $\Nb(Q)$ sont respectivement la valuation et l'ordre de $(Q \mod \varpi)$ dans $\Dx$.
\textit{L'exposant} de $\I$ en $x$ est défini par
\[ \Exp(\I) := \{ (N_k(Q) , \Nb(Q)) : Q \in \I_x \backslash \{0\}\} \subset \N^2 .\]
On définit comme pour un idéal de $\Dx$ une base de division de $\I$ relativement au point $x$.
C'est une famille d'éléments $(P_1 , \dots , P_r)$ de $\I_x$ échelonnée pour la fonction $\Nb$ telle que $P_i$ soit de fonction $N_k$ minimale parmi les éléments de même fonction $\Nb$, telle que $N_k(\I) = N_k(P_r)$ soit minimale parmi les éléments de $I$ et telle que $\Nb(\I) = \Nb(P_1)$ soit minimale parmi les éléments de $I$.
On demande de plus que les $P_i$ soient normalisés : pour tout $i \in \{ 1 , \dots , r\}$, $\| P_i\|_k = 1$.

\vspace{0.4cm}

Cette dernière condition permet d'assurer la compatibilité des bases de division dans $\Dk$ et dans $\Dx$ après réduction modulo $\varpi$.
En effet, soit $\I$ un idéal cohérent non nul de $\Dk$ admettant une base de division en $x$. On note $I$ la réduction modulo $\varpi$ de $\I$ ; c'est un idéal cohérent de $\Dks$ et $I_x$ est un idéal de $\Dx$.
Alors $(P_1 , \dots , P_r)$ est une base de division de $\I$ relativement à $x$ si et seulement si $(P_1 \mod \varpi , \dots , P_r \mod \varpi)$ est une base de division de $I_x$.
En particulier, les escaliers et les exposants de $I$ et $\I$ coïncident en $x$.

\vspace{0.4cm}

Les deux lemmes suivants sont démontrés pour $k = 0$ par Laurent Garnier dans \cite{garnier}, partie 4, proposition 4.2.1 et corollaire 4.2.2 respectivement.
Ils résultent de l'existence d'une division de tout élément de $\Dkq(U)$ par une base de division de $\I$.
Leurs preuves s'adaptent sans difficulté pour un niveau de congruence $k$ quelconque.

\begin{lemma}\label{lemmebase}
Toute base de division de $\I$ en $x$ engendre l'idéal $\I_x$.
\end{lemma}

Une base de division existe toujours pour les idéaux de $\Dkq$. Cependant ce n'est pas vrai dans $\Dk$. 
Si $\I$ est un idéal cohérent non nul de $\Dk$, il n'est pas toujours possible de normaliser les $P_i$ dans $\Dk(U)$.
En effet, de la $\varpi$-torsion peut poser problème. Le lemme suivant donne une condition nécessaire et suffisante pour que l'idéal $\I$ admette une base de division relativement au point $x$.

\begin{lemma}\label{lemme2.11}
Un idéal cohérent $\I$ non nul de $\Dk$ admet une base de division relativement au point $x$ si et seulement si $\Dk / \I$ est sans $\varpi$-torsion au voisinage de $x$.
\end{lemma}

La propriété suivante, à condition d'avoir une base de division, fournit une présentation finie d'un idéal cohérent à gauche de $\Dk$ en tant que $\Dk$-module à gauche. Il s'agit de la proposition 4.3.1 de \cite{garnier}.

\begin{prop}\label{propresolution}
On suppose que $x$ est un point $\kappa$-rationnel. Soit $\I$ un idéal cohérent non nul de $\Dk$ admettant une base de division $(P_1 , \dots , P_r)$ relativement à $x$.
Il existe un ouvert affine $U$ de $\X$ contenant $x$ et une matrice de relation $R \in \mathrm{M}_{r-1 , r}(\Dk(U))$ obtenue à partir des $P_i$ pour lesquels le complexe suivant de $\Dkq$-modules est exact  :
\[ \xymatrix{ 0 \ar[r] & (\widehat{\mathcal{D}}^{(0)}_{U, k})^{r-1} \ar[r]^{\cdot R} & (\widehat{\mathcal{D}}^{(0)}_{U, k})^r \ar[r] & \I_{|U} \ar[r] & 0 } .\]
\end{prop}

\section{$\Dkq$-modules holonomes}\label{partie4}

On commence cette section par définir la variété caractéristique d'un $\Dkq$-module cohérent.
Il s'agit d'une sous-variété fermée du fibré cotangent $T^*X$ de la fibre spéciale $X$ de $\X$.
On définit alors les $\Dkq$-modules holonomes comme étant les $\Dkq$-modules cohérents dont la variété caractéristique est de dimension au plus un.
On démontre dans la partie \ref{partie2.4} que les $\Dkq$-modules holonomes sont de longueur finie.
Cela découle de l'inégalité de Bernstein, établit dans la partie \ref{partie2.3} : un $\Dkq$-module cohérent est non nul si et seulement si les composantes irréductibles de sa variété caractéristique sont de dimension au moins un.
Cette inégalité généralise pour un niveau de congruence $k$ celle démontrée par Laurent Garnier dans l'article \cite{garnier}.
On en déduit que les multiplicités des variétés caractéristiques vont s'additionner dans la catégorie des $\Dkq$-modules holonomes et qu'un $\Dkq$-module cohérent est nul si et seulement ses multiplicités sont nulles.
On démontre enfin dans la partie \ref{partie3.5} que les $\Dkq$-modules holonomes vérifient la caractérisation cohomologique classique : $\Ext_{\Dkq}^d(\M , \Dkq) = 0$ pour tout entier $d \neq 1$.

\vspace{0.4cm}

On désigne toujours par $U$  un ouvert affine de $\X$ contenant le point fermé $x$ (non supposé $\kappa$-rationnel) sur lequel on dispose d'une coordonnée étale.

\subsection{Variété caractéristique des $\Dkq$-modules cohérents}\label{partie2.1}

On résume brièvement dans cette partie la construction de la variété caractéristique d'un $\Dkq$-module cohérent, adaptée de celle de Berthelot pour un indice de congruence $k$.
Cette variété est définie comme étant la variété caractéristique \og classique \fg \, de la réduction modulo $\varpi$ d'un $\Dk$-module cohérent, donc d'un $\Dks$-module cohérent.
Le lecteur peut consulter les notes de Berthelot, par exemple la partie 5.2 de \cite{berthelotintro}, pour plus de détails.

\vspace{0.4cm}

On rappelle que le faisceau $\Dks = \Dk \otimes_\V \kappa$ est la réduction modulo $\varpi$ de $\Dk$.
C'est un faisceau de $\kappa$-algèbres sur la fibre spéciale $X = \X \times_\V \Spec \kappa$ de $\X$.
Comme $\X$ et $X$ ont le même espace topologique, on peut identifier $U$ à un ouvert affine de $X$.
On note $\partial_k$ l'image de $\varpi^k \partial$ dans la $\kappa$-algèbre $\Dks(U)$. On a
\[ \D_{U , k} = \bigoplus_{n\in\N} \O_U\cdot \partial_k^n . \]
On munit le faisceau $\Dks$ de la filtration croissante donnée localement par l'ordre des opérateurs différentiels :
\[ \forall m \in \N, ~~ \Fil^m (\D_{U , k}) := \bigoplus_{n = 0}^m \O_U\cdot \partial_k^n . \]
On note $\gr \Dks = \bigoplus_{m\in \N} \gr_m \Dks$ le gradué associé et $\xi_k$ l'image de $\partial_k$ dans $\gr_1 (\Dks(U))$.
Localement, $\gr (\D_{U , k}) \simeq \O_U[\xi_k]$ est un anneau de polynômes en une variable sur $\O_U$.
En particulier, le fibré cotangent $T^* X$ de $X$ est isomorphe en tant que $\kappa$-schéma à $\Spec \gr (\Dks)$.
On identifie ces deux schémas dans la suite. On note $\pi : T^* X \to X$ la projection canonique.

\vspace{0.4cm}

Soit $P = \sum_{n=0}^d a_n \cdot \partial_k^n$ un opérateur différentiel de $\D_{X, k}(U)$ d'ordre $d$.
On lui associe un élément du gradué $\gr \Dks(U)$ appelé \textit{symbole principal} de $P$ en posant
\[ \sigma(P) := a_d \cdot \xi_k^d \in \gr_d\D_{X, k}(U) .\]

\begin{remark}
 On a $[ \varpi^k \partial , x ] = \varpi^k \cdot \id$ dans l'algèbre $\Dk(U)$.
 Pour tout entier $k \geq 1$, on a donc $[ \partial_k , x ] = 0$.
 Ainsi, $\Dks(U)$ est une $\kappa$-algèbre commutative, donc une algèbre de polynômes en une variable : $\Dks(U) = \O_X(U)[\partial_k]$.
\end{remark}

Une filtration $(\fil^\ell E)_{\ell\in\N}$ d'un $\Dks$-module quasi-cohérent à gauche $E$ est une suite croissante $(\fil^\ell E)_\ell$ de sous $\O_X$-modules quasi-cohérents de $E$ telle que
\begin{enumerate}
\item
$ E = \bigcup_{\ell\in \N} \fil^\ell E$ ;
\item
$\forall n , \ell \in \N$, $(\fil^n \Dks) \cdot (\fil^\ell E) \subset \fil^{\ell + n} E$.
\end{enumerate}

Le gradué $\gr E$ pour une telle filtration est un $\gr \Dks$-module. La filtration est appelée \textit{bonne filtration } si $\gr E$ est un $\gr\Dks$-module cohérent.
Puisque la courbe $X$ est quasi-compacte, tout $\Dks$-module cohérent admet une bonne filtration globale.

On considère maintenant un $\Dks$-module cohérent $E$ muni d'une bonne filtration globale.
On associe à $E$ le $\O_{T^* X}$-module cohérent
\[ \tilde{E} := \O_{T^*X} \otimes_{\pi^{-1} (\gr \Dks)} \pi^{-1} (\gr E) .\]

\begin{definition}
La variété caractéristique de $E$ est le support de $\tilde{E}$ : $\Car E := \supp \tilde{E}$.
\end{definition}

C'est une sous-variété fermée de $T^*X$ puisque le $\Dks$-module $\tilde{E}$ est cohérent.
La variété caractéristique est indépendante du choix de la bonne filtration choisie.
Ce résultat a par exemple été démontré dans le lemme D.3.1 de \cite{hotta} dans le cas d'un anneau commutatif noetherien, hypothèses que vérifie la $\kappa$-algèbre commutative $\gr \Dks(U)$.

\vspace{0.4cm}

On appelle \textit{multiplicités} de $E$ les multiplicités des composantes irréductibles de sa variété caractéristique $\Car E$.
Soit $C$ une composante irréductible de $\Car E$ et $\eta$ son point générique.
Par définition, la multiplicité $m_C = m_C(E)$ de $C$ est la longueur du $(\O_{T^*X})_\eta$-module artinien $\tilde{E}_\eta$.
C'est un entier positif non nul dès que le module $E$ est non nul. Lorsque $C$ est une partie fermée irréductible non vide du fibré cotangant $T^*X$ non contenue dans la variété caractéristique $\Car E$, on pose $m_C = 0$.

On note $I(\Car E)$ l'ensemble des composantes irréductibles de la variété caractéristique de $E$. On définit le \textit{cycle caractéristique} de $E$ par la somme formelle
\[ \cc(E) := \sum_{C \in I(\Car E)} m_C \cdot C .\]

\vspace{0.4cm}

On dispose du résultat suivant classique pour une variété complexe.
La preuve de ce dernier est faite par exemple dans le lemme D.3.3 de \cite{hotta}.
On l'utilisera plus tard pour démontrer que les $\Dkq$-modules holonomes sont de longueur finie.

\begin{prop}\label{propse}
Soit $\xymatrix{ 0 \ar[r] & M \ar[r] & N \ar[r] & L \ar[r] & 0 }$ une suite exacte de $\Dks$-modules cohérents. Alors $\Car N = \Car M \cup \Car L$.
De plus, si $C$ est une composante irréductible de $\Car N$, alors $m_C(N) = m_C(M) + m_C(L)$ (avec $m_C(M) = 0$ ou $m_C(L) = 0$ si $C$ n'est pas dans $\Car M$ ou $\Car L$).
\end{prop}

Soit maintenant $\E$ un $\Dkq$-module cohérent à gauche.
Un \textit{modèle entier} de $\E$ est un $\Dk$-module cohérent $\Eo$ sans $\varpi$-torsion tel que $\E \simeq \Eo \otimes_\V K$.
Puisque $\E$ est cohérent, il existe un modèle entier $\Eo$ d'après \cite{berthelot1}.
La réduction $\Eo  \otimes_\V \kappa$ modulo $\varpi$ de $\Eo$ est un $\Dks$-module cohérent.

\begin{definition}
La variété caractéristique de $\E$ est la variété caractéristique du $\Dks$-module cohérent $\Eo  \otimes_\V \kappa$ : $\Car \E := \Car (\Eo  \otimes_\V \kappa)$.
\end{definition}

C'est un sous-schéma fermé du fibré cotangent $T^*X$ de $X$ indépendant du choix du modèle entier.
On appelle \textit{multiplicités} de $\E$ les multiplicités de sa variété caractéristique.

\vspace{0.4cm}

On termine cette partie par des exemples explicites de variétés caractéristiques. Ils permettent en pratique de calculer toutes les variétés caractéristiques.

\begin{example}~
On suppose que la courbe formelle $\X$ est affine munie d'une coordonnée locale.
On note toujours $\xi_k = \sigma(\partial_k)$ l'image de la dérivation $\partial_k$ dans $\gr_1 \Dks(X)$.
\begin{enumerate}
\item
Puisque le support de $\Dks$ est $X$ tout entier, on a $\Car \Dkq = T^* X$.
\item
Si $\E = 0$, alors sa variété caractéristique est vide.
\item
Soit $\E = \Dkq / P$ avec $P \in \Dkq(\X)$ un opérateur différentiel non nul. Quitte à multiplier $P$ par une bonne puissance de $\varpi$, on peut supposer que $\| P \|_k= 1$. Alors $\Eo = \Dk / P$ est un modèle entier de $\E$.
On note $d = \Nb(P)$ et $b$ le coefficient d'indice $d$ de $P$. La réduction $\bar{P}$ de $P$ modulo $\varpi$ est un opérateur de $\Dks(X)$ d'ordre $d$. Son coefficient dominant est $\bar{b} = (b \mod \varpi) \in \O_{X}(X)$.

On munit $\Eo \otimes_\V \kappa \simeq \Dks / \bar{P}$ de la filtration quotient. On a $\gr (\Eo \otimes_\V \kappa )= \gr \Dks / (\sigma(\bar{P}))$, où $\sigma(\bar{P}) = \bar{b} \cdot \xi_k^d $ est le symbole principal de $\bar{P}$. 
L'annulateur de ce module est l'idéal engendré par $\sigma(\bar{P})$. La variété caractéristique de $\E$ est donnée par l'équation
\[ \Car(\E) = \{ (x, \xi) \in T^*X : \sigma(\bar{P})(x,\xi) = \bar{b}(x) \cdot \xi^d = 0 \} . \]
\item
Plus généralement, soit $\E = \Dkq / \I$ pour un idéal cohérent non nul $\I$ de $\Dkq$. On se donne un modèle entier $\Eo = \Dkq / \overset{\circ}{\I}$ de $\E$. On note $I$ la réduction modulo $\varpi$ de $\overset{\circ}{\I}$. C'est un idéal de $\Dks$. Alors
\[ \Car(\E) = \{ (x, \xi) \in T^*X : \sigma(P)(x , \xi) = 0 ~~ \forall P \in I \} . \]
\end{enumerate}
\end{example}

\subsection{Réduction au cas des $\Dx$-modules cohérents}\label{partiedx}

Pour $x \in X$, on note $\Dx := (\Dks)_x = \bigoplus_{n \in \N} \O_{X, x} \cdot \partial_k$.
On associe à tout $\Dx$-module cohérent $E$ une variété caractéristique $\Car E$ définie comme une sous-variété fermée du $\kappa$-schéma $\Spec \left(\gr (\Dx)\right)$.
On munit comme dans la partie précédente l'algèbre $\Dx$ de la filtration donnée par l'ordre des opérateurs différentiels.
On note $\gr \Dx$ le gradué associé. Le module $E$ admet une bonne filtration ; on note $\gr E$ le gradué correspondant.
Soit $\phi  :  \Spec \left(\gr (\Dx)\right) \to \Spec( \O_{X,x})$ la projection donnée par l'inclusion de $\O_{X,x}$ dans $\gr (\Dx)$.
Alors
\[ \Car E := \supp \left( \O_{\Spec \left(\gr (\Dx)\right)} \otimes_{\phi^{-1} (\gr \Dx)} \phi^{-1} (\gr E) \right) . \]
Cette variété caractéristique ne dépend pas de la bonne filtration choisie sur $E$.
On note $s : X \to T^* X$ la section nulle du fibré cotangent.
Les notions de variétés caractéristiques et de multiplicités coïncident entre les germes de $\Dkq$ en $x$ et $\Dx$.

\begin{lemma}\label{lemmex}
Soit $\E$ un $\Dkq$-module cohérent et $\Eo$ un modèle entier. On dispose d'un isomorphisme de $\kappa$-schémas
\[ \Car(\Eo \otimes \kappa) \times_X \Spec(\O_{X,x}) \simeq \Car(\Eo_x \otimes \kappa) . \]
De plus, les multiplicités de $\Car(\Eo_x  \otimes \kappa)$ sont les multiplicités des composantes irréductibles de $\Car(\Eo \otimes \kappa)$ contenant $s(x)$.
\end{lemma}
\begin{proof}
Le morphisme $\Dks(X) \otimes_{\O_X(X)} \O_{X,x} \to  \Dx = \bigoplus_{n \in \N} \O_{X,x} \cdot \partial_k^n$ est un isomorphisme de $\O_{X,x}$-modules car $\Dks = \bigoplus_{n \in \N} \O_X \cdot \partial_k^n$.
Il s'agit en fait d'un isomorphisme de $\kappa$-algèbres pour le produit sur $\Dks(X) \otimes_{\O_X(X)} \O_{X,x} $ induit par le produit tensoriel.
On dispose donc d'un isomorphisme de $\kappa$-algèbres :
\begin{equation}\label{isox}
\Dks(X) \otimes_{\O_X(X)} \O_{X,x} \simeq \Dx 
\end{equation}

On note $E = \Eo \otimes_\V \kappa$. C'est un $\Dks$-module cohérent.
La question étant locale en $x$, on peut supposer $X$ affine.
Comme $E$ est $\Dks$-cohérent, $E$ est un $\O_X$-module quasi-cohérent. Il est donc suffisant d'étudier le module des sections globales $E(X)$.

On munit le $\Dks(X)$-module $E(X)$ d'une bonne filtration $(\fil^n(E(X)))_{n\in \N}$, le module $E(X) \otimes_{\O_X(X)} \O_{X,x}$ de la filtration $\fil^n(E(X)) \otimes_{\O_X(X)} \O_{X,x}$ et $E_x$ de la filtration image. 
On note $\gr(E(X))$ et $\gr(E_x)$ les gradués associés.
Le problème se réduit donc à démontrer que $\gr(E(X)) \otimes \O_{X,x} \simeq \gr(E_x)$ en tant que $\gr( \Dx)$-modules et que les supports coincident.

Puisque $\O_{X,x}$ est le localisé de $\O_X(X)$ en $x$, $E(X) \otimes_{\O_X(X)} \O_{X,x}$ est isomorphe à $E_x$ en tant que $\O_{X,x}$-module.
Le morphisme $E(X) \otimes_{\O_X(X)} \O_{X,x} \simeq E_x$ est en fait un $\Dx$-isomorphisme.
En effet, l'isomorphisme \ref{isox} montre que l'on peut munir $E(X) \otimes_{\O_X(X)} \O_{X,x}$ d'une structure naturelle de $\Dx$-module.
On vérifie ensuite que cette structure coïncide avec celle de $E_x$ et que le morphisme $E(X) \otimes_{\O_X(X)} \O_{X,x} \simeq E_x$ est un $\Dx$-isomorphisme.

On en déduit que $\gr(E(X)) \otimes_{\O_X(X)} \O_{X,x} \simeq \gr(E_x)$ en tant que $\gr(\Dx(X))$-modules. Il reste à vérifier que leurs supports coïncident. Soit $y \in \Spec(\O_{X,x})$. On a
\[ \left( \gr(E(X)) \otimes_{\O_X} \O_{X,x} \right)_y = (\gr E(X))_{\varphi^{-1}(y)} \otimes_{\O_{X , \varphi^{-1}(y)}} (\O_{X,x})_y \]
où $\varphi$ est le morphisme canonique  $\O_X(X) \to \O_{X,x}$. Comme la $\kappa$-algèbre $\O_X(X)$ est intègre, ce module est non nul si et seulement si $(\gr E(X))_{\varphi^{-1}(y)}$ est non nul.
\end{proof}

\begin{remark}
On a demontré que $\Dks \otimes_{\O_X} \O_{X,x} \simeq  \Dx$ en tant que $\kappa$-algèbres. On identifie par la suite ces deux algèbres.
\end{remark}

On désigne maintenant le $\Dks$-module $\Eo \otimes_\V \kappa$ par $\E \otimes \kappa$ et le $\Dx$-module $\Eo_x  \otimes_\V \kappa$ par $\E_x \otimes \kappa$.
Ces notations sous-entendent le choix d'un modèle entier. Puisque la variété caractéristique ne dépend pas du modèle entier, les variétés $\Car(\E \otimes \kappa)$ et $\Car(\E_x \otimes \kappa)$ sont définies sans ambiguïté.

Lorsque $x$ est un point $\kappa'$-rationnel pour une extension finie $\kappa'$ de $\kappa$, il sera parfois nécessaire d'étendre les scalaires à $\kappa'$.
Cependant, si $E$ est un $\Dx$-module cohérent, les variétés caractéristiques de $E$ et $E \otimes_\kappa \kappa'$ auront la même dimension puisque l'extension $\kappa' / \kappa$ est finie. Il est donc suffisant de tout démontrer au niveau de $\kappa$.

\begin{definition}
On appelle multiplicités de $\E$ en $x$ les multiplicités de la variété caractéristique $\Car(\E_x  \otimes \kappa)$.
\end{definition}

D'après le lemme \ref{lemmex}, il s'agit des multiplicités des composantes irréductibles de la variété caractéristique de $\E$ contenant $s(x)$.

\vspace{0.4cm}

L'étude de la variété caractéristique d'un $\Dkq$-module cohérent se ramène donc à étudier les variétés caractéristiques des $\Dx$-modules cohérents.
On explicite dans ce paragraphe la variété caractéristique d'un $\Dx$-module cohérent non nul $E$. 
On peut tout d'abord se ramener au cas où $X$ est affine et $E = \Dx / I$ pour un idéal à gauche $I$ de $\Dx$.
En effet, puisque $E$ est cohérent, $E$ est engendré par des sections globales $e_1 , \dots , e_r$. Si $I_i = \ann_{\Dx}(e_i)$, alors $\Dx \cdot e_i \simeq \Dx /I_i$.
Comme la variété caractéristique est un support et puisque le support d'une somme est l'union des supports des termes de la somme, on a
\[ \Car(E) = \bigcup_{i=1}^r \Car(\Dx / I_i) .\]

Ainsi, on peut supposer que $E = \Dx /I$. Si $I = 0$, alors $\Car E = \Spec \left(\gr (\D_{\O_{X,x},k})\right)$ car le support de $\Dx$ est l'espace tout entier.
On considère maintenant le cas où $I$ est un idéal non nul. Soit $P_1  \dots , P_r$ une base de division de $I$ comme définie dans la partie \ref{partie1.3}.
Les symboles principaux $\sigma(P_1) , \dots , \sigma(P_r)$ engendrent le gradué $\gr(I)$ comme $\gr(\Dx)$-module.
On note $d = d(I)$ et $\alpha = v(I)$. Par définition, le couple $(\alpha , d)$ est l'exposant de $I$. On écrit $\Exp(P_1) = ( d , \alpha_1)$, $\Exp(P_2) = (d+1 , \alpha_2)$, \dots \, , $\Exp(P_r) = (d+r-1 ,\alpha)$ où $\alpha_1 \geq \alpha_2  \geq \dots \geq \alpha$.
Quitte à normaliser les opérateurs $P_i$, on a $\sigma(P_i) = t^{\alpha_i} \cdot \xi_k^{d+i-1}$.
La variété caractéristique de $E$ est alors

\[ \Car (E) = \left\{ (t , \xi_k) \in \Spec \gr(\Dx) : ~~ t^{\alpha_1} \cdot \xi_k^{d} = t^{\alpha_2} \cdot \xi_k^{d+1} = \dots = t^{\alpha} \cdot \xi_k^{d+r-1} = 0 \right\} .\]

Dans $\Dx$, la condition $I \neq 0$ n'est pas équivalente à la condition $\alpha \neq 0$ ou $d \neq 0$. On peut avoir $ \alpha =d =  0$ : c'est le cas par exemple pour $I = (t^n , \partial_k^\ell)$, $n , \ell \in \N$.
Les équations de la variété caractéristique de $E = \Dx / I$ se réduisent aux équations suivantes :
\begin{equation}\label{equacar}
 \Car (\Dx / I) = \begin{cases}
t \cdot \xi_k = 0~~ \mathrm{si} ~~ d(I) \neq 0 ~~ \mathrm{et} ~~ v(I) \neq 0  \\
\xi_k  = 0 ~~~  \mathrm{si} ~~ v(I)  = 0  \\
t = 0 ~~~~  \mathrm{si} ~~ d(I) = 0 \\
t = 0 ~~\mathrm{et}~~ \xi_k =0 ~~ \mathrm{si} ~~ d(I) = 0 ~~\mathrm{et} ~~ v(I) = 0
\end{cases} 
\end{equation}

Lorsque $\dim(\Car E) = 1$, la variété caractéristique de $E$ admet une ou deux composantes irréductibles données par les équations $t = 0$ et $\xi_k = 0$. Lorsque $\dim ( \Car E ) =0$, $\Car E = (0,0)$.
En particulier, l'inégalité de Bernstein est fausse pour les $\Dks$-modules cohérents.
Cependant, si $E$ provient d'un $\Dkq$-module cohérent, on montrera que le dernier cas de (\ref{equacar}) n'est pas possible.
 La variété caractéristique de $E$ sera donc donnée par l'une des trois premières équations.

\begin{example}\label{exempledirac}~
\begin{enumerate}
\item
Si $E = \Dx / (t^\alpha \cdot \partial_k^d)$ avec $\alpha , d \geq 1$, alors $\Car E$ a deux composantes irréductibles d'équations respectives $t = 0$ et $\xi_k = 0$.
\item
Si $E = \Dx / (t^n ,\partial_k^\ell)$, alors $\Car E = (0,0)$.
\item
Soit $E = \Dks / x$ un $\Dks$-module supporté en $x$. La variété caractéristique de $E$ en $x$ est la droite d'équation $t = 0$.
Soit $U$ un ouvert affine de $X$ contenant $x$ sur lequel on dispose d'une coordonnée locale $y$.
Le module $E$ étant nul en dehors de $U$, on peut supposer que $X = U$.
Alors $T^*X$ est affine et l'on note $(y , \xi)$ le système de cordonnées locales de $T^*X$ associé à la coordonnée initiale $y$.
On a $\Car(E) = \{ (y, \xi) \in T^*X : y = x \}$. La variété caractéristique de $E$ est la droite verticale de $T^*X$ passant par $x$ :

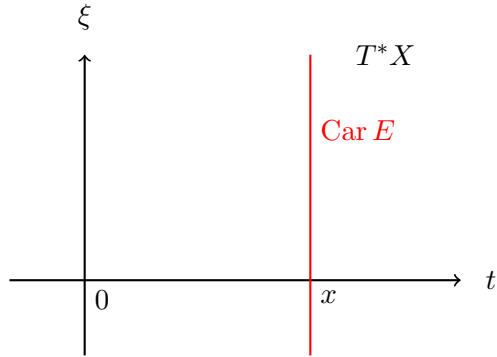
\begin{center}
\begin{tikzpicture}

\draw[thick][->] (-1,0) -- (5,0);
\draw (5.2,0) node[right] {$y$};
\draw [thick][->] (0,-1) -- (0,3);
\draw (0,3.2) node[above] {$\xi$};
\draw (0,0) node[below right] {$0$};

\draw[red][thick] (3, -1) -- (3,3) ;
\draw (3,2) node[right]{\textcolor{red}{$\Car E$}} ;
\draw (3,0) node[below right]{$x$} ;
\draw (4,3) node{$T^*X$} ;

\end{tikzpicture}
\end{center}
Un tel module est appelé un \textit{Dirac}.
\end{enumerate}
\end{example}

Un $\Dx$-module de la forme $\Dx / I$ distinct de $\Dx$ a deux multiplicités correspondant aux composantes $t = 0$ et $\xi_k = 0$, avec multiplicité nulle si la composante est un point ou si la composante est vide.
Lorsque $x$ est un point $\kappa$-rationnel, ces multiplicités correspondent aux nombres $d(I)$ et $v(I)$. Cela a été prouvé par P.Maisonobe dans \cite{maisonobe}, partie III, paragraphe 2.1.

\vspace{0.4cm}

Soit maintenant $\E = \Dk / \I$ pour un idéal cohérent non nul $\I$. D'après ce que l'on vient de dire, $\E$ a deux multiplicités en $x$ (potentiellement nulles) correspondant aux composantes $t = 0$ et $\xi_k = 0$ de la variété caractéristique $\Car(\E_x \otimes \kappa)$.
Ces multiplicités en un point rationnel sont respectivement les fonctions $\Nb(\I)$ et $N_k(\I)$.

\begin{prop}\label{prop3.10}
Soit $x$ un point $\kappa$-rationnel et $\I$ un idéal cohérent non nul de $\Dk$ tel que $\E = \Dk / \I$ soit sans $\varpi$-torsion.
Alors $\Nb(\I)$ et $N_k(\I)$ sont les multiplicités de $\E$  en $x$ des composantes $(\xi_k = 0)$ et $(t = 0)$ de la variété caractéristique $\Car(\E_x \otimes \kappa)$.
\end{prop}
\begin{proof}
Puisque le module $\E = \Dk / \I$ est sans $\varpi$-torsion, $\I$ admet une base de division dans $\Dk$ d'après le lemme \ref{lemmebase}.
L'énoncé étant local en $x$, on peut supposer $\X$ affine. La suite exacte courte $\xymatrix{ 0 \ar[r] & \I \ar[r] & \Dk \ar[r] & \E \ar[r] & 0}$ permet d'obtenir la suite exacte suivante :
\[ \xymatrix{ 0 \ar[r] & \mathrm{Tor}_\V ^1(\E\otimes_\V\kappa) \ar[r] & \I \otimes_\V \kappa \ar[r] & \Dk \otimes_\V \kappa \ar[r] & \E\otimes_\V \kappa \ar[r] & 0 } .\]
Par hypothèse, $\E = \Dk / \I$ est sans $\varpi$-torsion. On en déduit que $\mathrm{Tor}_\V ^1(\E\otimes_\V\kappa) = 0$.
On obtient un $\Dx$-isomorphisme
\[ \E_x \otimes_\V \kappa \simeq \Dx /(\I \otimes_\V \kappa)_x .\]
Ainsi, le module $\E_x \otimes_\V \kappa$ est donné par l'idéal $I = \I \otimes_\V \kappa$. On rappelle que $\I$ et $I$ ont le même escalier en $x$ et que les fonctions $\Nb(\I)$ et $N_k(\I)$ coïncident avec $v(I)$ et $d(I)$.
Comme les multiplicités de $\Dx / I$ sont respectivement l'ordre et la valuation de l'idéal $I$ en $x$, on obtient le résultat.
\end{proof}

Soit enfin $\E = \Dkq / \I$ pour un idéal cohérent $\I$ non nul de $\Dkq$.
Si $\E \simeq \Dkq / \I'$ pour un autre idéal $\I'$, alors la proposition 4.2.1 de l'article \cite{garnier} de Laurent Garnier (division selon une base de division) implique que $\I$ et $\I'$ ont les mêmes fonctions $\Nb$ et $N_k$.
Les entiers $\Nb(\I)$ et $N_k(\I)$ ne dépendent donc pas du choix de l'idéal $\I$ définissant $\E$ comme un quotient de $\Dkq$.
Lorsque $\I = \Dkq \cdot P$ avec $P \neq 0$, ces nombres sont simplement $\Nb(P)$ et $N_k(P)$.

On peut trouver un modèle entier de $\E$ de la forme $ \Dk /\J$.
C'est un $\Dk$-module cohérent sans $\varpi$-torsion tel que $\E \simeq (\Dk /\J) \otimes_\V K$.
D'après le lemme \ref{lemme2.11}, l'idéal $\J$ admet une base de division en chaque point $x \in \X$.
Puisque $(\Dk /\J) \otimes_\V K \simeq \Dkq / (\J \otimes_\V K)$, on a $ \Dkq / \I \simeq\Dkq / (\J \otimes_\V K)$. Ainsi, les idéaux $\I$ et $\J \otimes_\V K$ ont les mêmes fonctions $\Nb$ et $N_k$.
Enfin, puisque $\J$ et $\J \otimes_\V K$ ont les mêmes escaliers, on obtient $\Nb(\I) = \Nb(\J)$ et $N_k(\I) = N_k(\J)$. On en déduit le résultat suivant.

\begin{cor}
Soit $x$ un point $\kappa$-rationnel et $\E \simeq \Dkq / \I$ avec $\I$ un idéal cohérent non nul de $\Dkq$.
Alors $\Nb(\I)$ et $N_k(\I)$ sont respectivement les multiplicités de $\E$  en $x$ des composantes irréductibles $(\xi_k = 0)$ et $(t = 0)$ de $\Car(\E_x \otimes \kappa)$.
\end{cor}


\subsection{Inégalité de Bernstein}\label{partie2.3}

Cette partie est consacrée à la démonstration de l'inégalité de Bernstein : un $\Dkq$-module cohérent est non nul si et seulement si sa variété caractéristique est de dimension au moins un, ou de manière équivalente si ses multiplicités ne sont pas toutes nulles.

\vspace{0.4cm}

Comme on a pu le voir dans la partie précédente, l'inégalité de Bernstein est fausse pour $\Dx$-modules cohérents.
Par exemple, la variété caractéristique du $\Dx$-module $E = \Dx / (t^p , \partial_k)$ est réduite au point $(0 , 0)$.
L'inégalité de Bernstein étant vraie pour un $\Dkq$-module cohérent, cela signifie que $E$ ne provient pas d'un $\Dkq$-module cohérent.
On peut cependant remarquer que $E$ est un $\kappa$-espace vectoriel de dimension finie (égale à $p$).
Plus généralement, ce résultat est vrai pour tout $\Dx$-module cohérent dont la variété caractéristique est réduite à un point.

\begin{lemma}\label{lemmeev}
Soit $x \in X$ et $E$ un $\Dx$-module de type fini dont la variété caractéristique $\Car E$ est un point. Alors $E$ est un $\kappa$-espace vectoriel de dimension finie.
\end{lemma}
\begin{proof}
On traite tout d'abord le cas où $x$ est un point $\kappa$-rationnel. On se donne une bonne filtration $(\Fil^i E)_{i \in \N}$ de $E$. On note $\gr E$ le gradué associé.
Par définition d'une bonne filtration, il existe des éléments $e_1 , \dots , e_n$ de $E$ tels que leurs symboles principaux $\sigma(e_1), \dots , \sigma(e_n)$ engendrent le gradué $\gr E$.
On peut démontrer que $e_1 , \dots ,e_n$ engendrent $E$ en tant que $\Dx$-module.
On note $d_\ell$ l'ordre de $e_\ell$.
On peut vérifier que pour tout entier $m \geq \max_{1 \leq \ell \leq n} \{ d_\ell \}$, $\fil^m (E) = \sum_{\ell=1}^n \fil^{m-d_\ell} (\Dks) \cdot e_i$
On en déduit que $ \Fil^i (\Dx) \cdot \Fil^j (E) = \Fil^{i+j} (E)$ à partir d'un rang $j$ et que les $\Fil^i(E)$ sont des $\O_{X,x}$-modules de type fini.
Quitte à décaler la filtration, on peut supposer que
\[ \forall i , j \in \N, ~~ \Fil^i (\Dx) \cdot \Fil^j (E) = \Fil^{i+j} (E) .\]
En particulier, $\Fil^i (E) = \Fil^i (\Dx) \cdot \Fil^0 (E)$. Ainsi, tout système de générateurs $(e_1 , \dots, e_r)$ de $\Fil^0(E)$ en tant que $\O_{X,x}$-module engendre $E$ en tant que $\Dx$-module.

On suppose que $\Car E$ est un point. L'idéal définissant $\Car E$ est un idéal maximal de $\O_{X,x}[\xi_k]$ homogène en $\xi_k$ : le point $\Car(E)$ correspond nécessairement à l'idéal $(t , \xi_k)$.
Il existe donc deux entiers $d$ et $v$ tels que $t^v$ et $\xi_k^d$ annulent $\gr E$.
En particulier, $\xi_k^d \cdot \gr^i(E)$ est nul dans $\gr^{d+i}(E)$. Sur la filtration cela se traduit par
\[ \partial_k ^d \cdot \Fil^i E \subset \Fil^{i+d-1} (E) = \Fil^{d-1} (\Dx) \cdot \Fil^i (E) .\]
Pour $i=0$, on obtient
\[ \partial_k^d \cdot \Fil^0 (E) \subset \Fil^{d-1} (\Dx) \cdot \Fil^0(E) .\]
Il en résulte que pour tout entier naturel $i$,
\[ \Fil^i (E) = \Fil^i(\Dx) \cdot \Fil^0(E) \subset \Fil^{d-1} (\Dx) \cdot \Fil^0(E) .\]
La filtration de $E$ est donc stationnaire et $\Fil^n(E) = E$ pour tout entier $n \geq d-1$.
Ainsi, $E$ est engendré sur $\O_{X,x}$ par les $\partial_k^j \cdot e_i$ pour $j\in \{ 0 , \dots , d-1 \}$ et $i \in \{1, \dots , r\}$ : $E$ est un $\O_{X,x}$-module de type fini.

On rappelle que $\Fil^0(E)$ est annulé par $t^v$ et que pour tout entier $k\geq 1$, l'algèbre $\Dx$ est commutative.
Puisque $\Fil^0(E)$ engendre $E$ en tant que $\Dx$-module, $E$ est annulé par $t^v$ dès que $k \geq 1$.
Sinon, lorsque $k = 0$, le fait que $t^v$ annule $\gr E$ implique que $t^{v(\ell+1)}$ annule $\fil^\ell (E)$.
En particulier, $t^{v d}$ annule $E = \Fil^{d-1}(E)$. Dans tous les cas, $E$ est annulé par une puissance de $t$ que l'on note encore $t^v$.

Ainsi, $E = E / t^v E$ est un $\O_{X,x}/t^v \O_{X,x}$-module de type fini.
Pour conclure, il suffit de prouver que $\O_{X,x}/t^v \O_{X,x}$ est un $\kappa$-espace vectoriel de dimension finie. On le montre par récurrence sur $v$. On dispose de la suite exacte d'anneaux
\[ \xymatrix{ 0 \ar[r] & t^{v-1} \O_{X,x} / t^v \O_{X,x} \ar[r] & \O_{X,x}/t^v \O_{X,x} \ar[r] & \O_{X,x}/t^{v-1} \O_{X,x} \ar[r] & 0 } \]
avec $t^{v-1} \O_{X,x} / t^v \O_{X,x} \simeq \kappa = \O_{X,x}/t \O_{X,x}$ (via la multiplication par $ t^{v-1}$).
La première flèche $\kappa \to \O_{X,x} / t^v \O_{X,x}$ munit $\O_{X,x} / t^v \O_{X,x}$ d'une structure de $\kappa$-espace vectoriel.
La suite reste exacte en considérant les quotients comme des $\kappa$-espaces vectoriels.
Par hypothèse de récurrence, $\O_{X,x}/t^{v-1} \O_{X,x}$ est un $\kappa$-espace vectoriel de dimension finie.
Ainsi, $\O_{X,x}/t^v \O_{X,x}$ est aussi de dimension finie sur $\kappa$.

Si maintenant $x$ est un point quelconque, alors $x$ est $\kappa'$-rationnel pour une extension finie $\kappa'$ de $\kappa$. Le même raisonnement montre que $E$ est un $\kappa'$-espace vectoriel de dimension finie. Puisque $\kappa'$ est de dimension finie sur $\kappa$, $E$ sera un $\kappa$-espace vectoriel de dimension finie.
\end{proof}

\begin{prop}[inégalité de Bernstein]\label{bernstein}
Soit $\E$ un $\Dkq$-module cohérent non nul. Alors toute composante irréductible de $\Car \E$ est de dimension au moins un.
En particulier, $\dim(\Car \E) \geq 1$. De plus les multiplicités de $\E$ sont non nulles. 
\end{prop}
\begin{proof}
On note $E = \E \otimes \kappa$ la réduction modulo $\varpi$ d'un modèle entier $\Eo$ de $\E$.
On rappelle que par définition, $\Car \E = \Car E$. Si $\E$ est non nul, alors $E$ est aussi non nul. Dans ce cas, $\Car \E \neq \emptyset$.

On suppose qu'une composante irréductible de $\Car \E$ est un point $z = (x , \xi)$. Alors $\Car \E_x = \Car E_x$ est contenue dans un point.
Si cette variété caractéristique est vide, alors $\E_x = 0$ et $E_x = 0$.
Sinon le lemme \ref{lemmeev} montre que $E_x$ est un $\kappa$-espace vectoriel de dimension finie.

On en déduit que $\E$ est un $K$-espace vectoriel de dimension finie au voisinage de $x$.
En effet, soit $\bar{e}_1 , \dots, \bar{e}_r$ une base de $E_x$ comme $\kappa$-espace vectoriel.
On note $e_1 , \dots , e_r$ des relèvements de ces éléments dans $\Eo_x$ et $\F = \V \cdot e_1 + \dots + \V \cdot e_r$.
C'est un sous-$\V$-module complet de $\Eo_x$ pour la topologie $\varpi$-adique. Soit $y \in \Eo_x$. On montre que $y \in \F$.
Puisque $\bar{y} \in E = \kappa \cdot \bar{e}_1 + \dots + \kappa\cdot \bar{e}_r$, il existe $y_1 \in \F$ et $z_1 \in \varpi \cdot \Eo_x$ tels que $y = y_1 + z_1$.
De même, $\varpi^{-1} z_1$ s'écrit sous la forme $y_2 + \tilde{z}_2$ avec $y_2 \in \F$ et $\tilde{z}_2 \in \varpi\cdot \Eo_x$.
On obtient $y =  (y_1 + \varpi  y_2) + z_2$ avec $y_1 , y_2 \in \F$ et $z_2 = \varpi  \tilde{z}_2 \in \varpi^2 \cdot \Eo_x$. Une récurrence montre que pour tout entier $n \geq 1$, il existe $y_1, \dots , y_n \in \F$ et $z_n \in \varpi^n \cdot \Eo_x$ tels que
\[ y = y_1 + \varpi y_2 + \dots + \varpi^{n-1} y_n + z_n .\]
Puisque $\F$ est complet pour la topologie $\omega$-adique, le terme $y_1 + \varpi y_2 + \dots + \varpi^{n-1} y_n$ converge vers un élément $y_\infty \in  \F$.
Par ailleurs, comme $\Eo$ est sans $\varpi$-torson, $\Eo$ est séparé pour la topologie $\varpi$-adique. Ainsi, la suite $(z_n)_n$ converge vers zéro.
Le passage à la limite $n \to \infty$ donne $y = y_\infty \in  \F$.
Autrement dit, $\Eo_x = \F = \V \cdot e_1 + \dots + \V \cdot e_r$.
On en déduit donc que $\E_x \simeq \Eo_x \otimes_\V K = K\cdot e_1 + \dots + K \cdot e_r$ est un $K$-espace vectoriel de dimension finie.

On rappelle que $[ \varpi^k \partial , t] = \varpi^k \cdot  \id$. Comme $\E_x$ est un $K$-espace vectoriel de dimension finie, on a
\[ \Tr ( [ \varpi^k \partial , t]  ) = 0 = \Tr(\varpi^k \cdot  \id) = \varpi^k \cdot  \Tr(\id) = (\varpi^k \dim_K \E_x) .\]
Puisque $K$ est de caractéristique nulle, $\dim_K \E_x = 0$. Donc $\E_x = 0$ et $\E$ est nul au voisinage de $x$.

Dans tous les cas, $E_x = 0$ et $E$ est nul au voisinage de $x$.
Ainsi, le support de $E$ est un sous-schéma fermé propre de $X$ : sa dimension est strictement inférieure  à $\dim X = 1$ puisque $X$ est irréductible.
Le support de $E$ consiste donc en un nombre fini de points.
Autrement dit, $E$ est une somme directe de Dirac (ie de $\Dks$-modules supportés en un point).
Mais la variété caractéristique d'un Dirac est de dimension un, voir l'exemple \ref{exempledirac}.
La variété caractéristique de $E$ en $x$ est alors une union finie de droites d'après la proposition \ref{propse}.
Cela contredit l'hypothèse qu'une composante irréductible est un point.
Ainsi, soit $\E$ est nul, soit les composantes irréductibles de $\Car \E$ sont de dimension au moins un.

On rappelle que $\Car E = \supp \tilde{E}$ où $\tilde{E} = \O_{T^*X} \otimes_{\pi^{-1} (\gr \Dks)} \pi^{-1} (\gr E)$ est un $\O_{T^*X}$-module cohérent.
Soit $\eta$ le point générique d'une composante irréductible $C$ de $\Car \E$.
La multiplicité $m_C$ de $C$ est la longueur du $(\O_{T^*X})_\eta$-module artinien $\tilde{E}_\eta$.
Si $\E$ est non nul, alors $\tilde{E}_\eta$ est aussi non nul. Sa longueur $m_C$ est donc supérieure ou égale à un.
Autrement dit, les multiplicités des composantes irréductibles de $\Car \E$ sont toutes non nulles.
\end{proof}

\begin{cor}\label{cor3.14}
Un $\Dkq$-module cohérent $\E$ est nul si et seulement $\dim ( \Car \E) = 0$, ou de manière équivalente si toutes ses multiplicités sont nulles.
\end{cor}
\begin{proof}
Le premier point découle de la proposition \ref{bernstein}. On a vu que $\E \neq 0$ implique $\Car \E \neq  \emptyset$.
En particulier, si $\Car \E =  \emptyset$, alors $\E = 0$.
Dans ce cas, les multiplicités de $\E$ en les fermés irréductibles non vides de $T^*X$ sont nulles par définition.
Ainsi, $\E$ est nul si et seulement si ses multiplicités sont toutes nulles.
\end{proof}

\subsection{Modules holonomes}\label{partie2.4}

On démontre dans cette partie les propriétés vérifiées par les $\Dkq$-modules holonomes énoncées dans l'introduction.
En particulier, les $\Dkq$-modules holonomes coincident avec les $\Dkq$-modules de longueur finie.

\begin{definition}
Un $\Dkq$-module cohérent $\E$ est appelé module holonome si $\E = 0$ ou si $\dim \Car(\E) = \dim X = 1$.
\end{definition}

Par l'inégalité de Bernstein, un module $\E$ est holonome si et seulement si $\dim \Car \E \leq 1$.
La catégorie des $\Dkq$-modules holonomes est une sous-catégorie abélienne des $\Dkq$-modules cohérents d'après la proposition \ref{propse}. On réécrit ci-dessous son énoncé pour les $\Dkq$-modules cohérents.

\begin{prop}\label{propholo}
Soit $\xymatrix{ 0 \ar[r] & \M \ar[r] & \Nn \ar[r] & \L \ar[r] & 0 } $
une suite exacte de $\Dkq$-modules cohérents.
Alors $\Car \Nn = \Car \M \cup \Car \L$. En particulier, $\Nn$ est holonome si et seulement si $\L$ et $\M$ le sont.
\end{prop}

Voici un exemple de modules holonomes  : tout $\Dkq$-module cohérent de la forme $\Dkq / \I$ est holonome dès que $\I$ est un idéal cohérent non nul de $\Dkq$.

On regarde tout d'abord le cas très explicite où $\X = U$ est affine et $\E = \Dkq /P$ pour un opérateur différentiel $P$ non nul de $\Dkq(\X)$. On normalise $P$ afin d'avoir $\| P \|_k =1$.
Soit $\bar{P}$ l'image de $P$ dans $\Dks(X)$ et $x \in X$. On écrit $\bar{P} = \sum_{n=0}^d a_n \cdot \partial_k^n$  avec $d = \Nb(P)$.
On note $\alpha = N_k(P)$ la valuation de $a_d$ dans l'anneau de valuation discrète $\O_{X,x}$.
Quitte à multiplier $\bar{P}$ par un élément inversible de $\O_{X,x}$, on peut supposer que le coefficient dominant de $\bar{P}$ est $t^\alpha$.
Par définition, $(d , \alpha)$ est l'exposant de $P$ et de $\bar{P}$.
On note $E = \Dx / \bar{P}$.  Lorsque $x$ est $\kappa$-rationnel, les multiplicités de $E$ sont $d$ et $\alpha$.
L'idéal annulateur de $E$ est le radical de l'idéal engendré par le symbole principal $\sigma(\bar{P}) = t^\alpha \cdot \xi_k^d$ de $\bar{P}$.
On suppose $P$ non inversible au voisinage de $x$, ce qui est équivalent à avoir $\alpha \neq 0$ ou $d \neq 0$ d'après le corollaire \ref{corinv}.
Dans ce cas, $E$ est $\Dx$-module non nul. La variété caractéristique de $\E$ en $x$ est alors donnée par les équations
\[\Car E = \begin{cases}
t \cdot \xi_k = 0~~ \mathrm{si} ~~ \alpha \neq 0 ~~et ~~d \neq 0  \\
\xi_k  = 0 ~~~  \mathrm{si} ~~ \alpha = 0  \\
t = 0 ~~~~  \mathrm{si} ~~ d = 0
\end{cases} \]
Ces composantes irréductibles sont toutes de dimension 1 et $\dim \Car( \E_x \otimes \kappa) = 1$.
La variété caractéristique de $\E$ est donc de dimension $1$ et $\E$ est holonome.
Si $P$ est inversible au voisinage de $x$, alors $E = 0$ et la variété caractéristique de $\E$ en $x$ est vide. Cette condition est équivalente à $\alpha = d = 0$. On retrouve ainsi l'inégalité de Bernstein.

On passe au cas où $\E = \Dkq / \I$ pour un idéal cohérent $\I$ non nul. Soit $\Eo = \Dk / \J$ un modèle entier de $\E$.
La réduction modulo $\varpi$ de $\J$ est un idéal de $\Dks$ que l'on note $I$. L'exposant de $I_x$ est le couple $(N_k(\I) , \Nb(\I))$.
Le $\Dx$-module $E = \Dx / I_x$ est isomorphe à $\E_x \otimes \kappa$. Si $\E \neq 0$, alors $\E_x \otimes \kappa \neq 0$ pour au moins un point $x$ de $X$.
D'après les formules (\ref{equacar}) données page \pageref{equacar} et l'inégalité de Bernstein, on a
\[ \Car (\E_x \otimes \kappa) = \begin{cases}
t \cdot \xi_k = 0~~ \mathrm{si} ~~ N_k(\I) \neq 0 ~~et ~~ \Nb(\I) \neq 0  \\
\xi_k  = 0 ~~~  \mathrm{si} ~~ N_k(\I)  = 0  \\
t = 0 ~~~~  \mathrm{si} ~~ \Nb(\I) = 0 \\
\end{cases} \]
La variété caractéristique de $\E_x \otimes \kappa$ est donc de dimension 1. Si $\E  \neq 0$, on en déduit que $\dim \Car(\E) = 1$ et $\E$ est holonome.
Réciproquement, on verra plus tard que tout module holonome est de cette forme.

\begin{prop}\label{prop3.17}
Soit $\xymatrix{ 0 \ar[r] & \M \ar[r] & \Nn \ar[r] & \L \ar[r] & 0 }$ une suite exacte de $\Dkq$-modules holonomes. Alors $\cc(\Nn) = \cc (\M) + \cc(\L)$.
Autrement dit, les multiplicités s'additionnent pour les $\Dkq$-modules holonomes.
\end{prop}
\begin{proof}
La proposition \ref{propse} nous assure que $\Car \Nn = \Car \M \cup \Car \L$.
Elle nous dit aussi que si $C \in I( \Car \Nn)$ (ensemble des composantes irréductibles de $\Car \Nn$), alors $C \in I( \Car \M)$ ou $C \in I( \Car \L)$ et que $m_C(\Nn) = m_C(\M) + m_C(\L)$.
On suppose $\M$, $\L$ et $\Nn$ non nuls.
Alors $\dim \Car \Nn = \dim \Car \M = \dim \Car \L =1$ et toutes les composantes irréductibles sont de dimension un d'après l'inégalité de Bernstein.

Soit $I$ une composante irréductible de $\Car \M$ ou de $\Car \L$. Alors $I$ est un fermé irréductible de $\Car(\Nn)$ de dimension maximale $1 = \dim \Car \Nn$ : $C$ est donc une composante irréductible de $\Car \Nn$.
Ainsi, $I(\Car \Nn) = I(\Car \M) \cup I(\Car \L)$. L'égalité des cycles en découle puisqu'alors les multiplicités s'additionnent d'après \ref{propse}.
\end{proof}

\begin{remark}
Lorsque $\dim(\Car\Nn)) =2$, une composante irréductible de $\Car \L$ ou de $\Car \M$ n'est pas toujours une composante irréductible de $\Car \Nn$.
En effet, la dimension de la variété caractéristique $\Car \Nn$ peut être strictement supérieure à celle de $\Car \L$ ou de $\Car \M$.
Les multiplicités ne s'additionnent donc pas dans la catégorie des $\Dkq$-modules cohérents.
\end{remark}

On rappelle que $X$ est une courbe lisse connexe quasi-compacte. Le fibré cotangent $T^*X$ reste quasi-compact et noethérien.
La variété caractéristique de tout $\Dkq$-module cohérent a donc un nombre fini de composantes irréductibles et un nombre fini de multiplicités.
Puisque les multiplicités sont additives et puisqu'un module dont les multiplicités sont nulles est nul, tout $\Dkq$-module holonome va être de longueur finie.

\begin{prop}\label{prop2.18}
Tout $\Dkq$-module holonome est de longueur finie, inférieure à la somme de toutes ses multiplicités.
\end{prop}
\begin{proof}
Soit $\E$ un $\Dkq$-module holonome. Sa variété caractéristique a un nombre fini de composantes irréductibles et $\E$ n'a qu'un nombre fini de multiplicités.
Puisque $\Dkq$ est noetherien, il suffit de montrer que toute suite décroissante $(\E_n)_{n \in \N}$ de sous-$\Dkq$-modules de $\E$ est stationnaire.
On suppose que $\E_0 = \E$. Comme $\E_n$ est inclus dans $\E$, $\E_n$ est holonome. On considère la suite exacte courte de modules holonomes
\[ \xymatrix{ 0 \ar[r] & \E_{n+1} \ar[r] & \E_n \ar[r] & \E_n / \E_{n+1} \ar[r] & 0 }. \]
Les multiplicités de $\E_n$ sont la somme de celles de $\E_{n+1}$ et de $\E_n / \E_{n+1}$.
En particulier, les suites des multiplicités sont décroissantes. Elles sont donc stationnaires à partir d'un certain rang commun $n_0$ puisqu'il n'y a qu'un nombre fini fixé de multiplicités (donné par le nombre de multiplicités de $\E = \E_0$).
Pour tout entier $n \geq n_0$, les multiplicités de $\E_n / \E_{n+1}$ sont donc nulles par additivité.
Autrement dit, $\E_n / \E_{n+1} = 0$ par l'inégalité de Bernstein. Donc $\E_n = \E_{n+1}$ pour tout $n \geq n_0$.
Ainsi, $\E$ est de longueur finie inférieure ou égale à la somme de ses multiplicités.
\end{proof}

Le théorème suivant de Stafford, énoncé initialement pour les algèbres de Weyl, implique que tout $\Dkq$-module holonome est monogène. La preuve étant élémentaire, on en redonne une démontrée dans \cite{stafford}, partie 4.

\begin{theorem}[Stafford]\label{stafford}
Soit $R$ un anneau simple de longueur infinie en tant que $R$-module à gauche. Alors tout $R$-module de longueur finie est monogène.
\end{theorem}
\begin{proof}
Soit $M$ un $R$-module de longueur finie. On commence par démontrer que $M$ est engendré par deux éléments $\alpha$ et $\beta$ par récurrence sur la longueur $\ell $ de $M$.
Si $\ell = 1$, alors $M$ est simple et donc engendré par un élément. Soit $\alpha \in M \backslash \{0\}$. Si $ M \neq R \cdot \alpha$, alors $M / R \alpha \neq 0$.
Puisque $\ell(M / R \alpha ) < \ell$, l'hypothèse de récurrence implique que $M / R \alpha $ est engendré par un élément $\bar{\beta}$ pour un certain $\beta \in M$.
Alors $M$ est engendré par $\alpha$ et $\beta$ en tant que $R$-module : $M = R \alpha + R \beta$.
On suppose dans la suite que $R \alpha \nsubseteq R \beta$ et que $R \beta \nsubseteq R \alpha$.
Pour toute paire d'éléments $(x , y)$ de $M$, on note
\[ \ell(x , y) = ( \ell(R y) , \ell((R x + R y) / R x)) \in \N^2 . \]
On dit que $(x' , y') < (x,y)$ si $\ell(x',y') < \ell(x,y)$ pour l'ordre lexicographique.
On suppose par récurrence sur $\ell(\alpha,\beta) \in \N^2$ que pour tout couple $(\alpha' , \beta') < (\alpha , \beta)$, il existe $\gamma' \in M$ tel que $R \alpha' + R \beta' = R \gamma'$.
L'initialisation est donnée par $\ell(0 , 0) = (0 , 0)$ pour $M = 0$.

Puisque $\ell(R) = + \infty$, $L(\alpha) := \ann_R(\alpha) \neq 0$.
En effet, l'application $R \to R \alpha, ~ a \mapsto a \alpha$ n'est pas injective car $\ell(R\alpha) < \infty$.
On fixe un élément $f \in L(\alpha) \backslash \{0\}$. Comme $R$ est simple, on peut trouver des éléments $s_1 , \dots , s_m, r_1 , \dots , r_m \in R$ tels que $1 = \sum_{i=1}^m s_i \cdot f \cdot r_i$.

S'il existe $x \in L(\alpha)$ et $y \in L(\beta)$ tels que $1 = x r + y$ pour un certain $r = r_i$, alors $M$ est engendré par un élément.
En effet, $\beta = (xr+y) \beta = xr\beta =x (\alpha+r\beta)$ car $y\beta = x\alpha = 0$ et $\alpha = (\alpha+r\beta) - r\beta$.
Ainsi, $\alpha, \beta \in R \cdot (\alpha+r\beta)$ et $M = R \alpha + R \beta = R\cdot (\alpha+r\beta)$.
On considère maintenant le cas où $R \neq L(\beta) + L(\alpha)\cdot r_i$ pour tout $i \in \{ 1 , \dots , m\}$.

Puisque $\sum_{i=1}^m s_i \cdot f \cdot r_i = 1$, on a $\sum_{i=1}^m R \cdot f \cdot r_i = R$ et $\sum_{i=1}^m R \cdot f  r_i \beta= R \beta$.
Comme $R \beta \nsubseteq R \alpha$, il existe un élément $r = r_i$ tel que $R \cdot fr\beta \not\subset R \alpha$.

L'inclusion stricte $L(\beta) + R \cdot fr \subset L(\beta) + L(\alpha) \cdot r \varsubsetneq R$ implique
\[ R \cdot fr\beta \simeq (L(\beta) + R\cdot fr) / L(\beta) \varsubsetneq R / L(\beta) \simeq R \beta . \]
Autrement dit, $(\alpha , fr\beta) < (\alpha , \beta)$. Par hypothèse, il existe $\gamma' \in M$ tel que $R \gamma' = R\cdot fr\beta + R \alpha$.
Puisque $R \cdot fr\beta \not\subset R \alpha$, $R \alpha \varsubsetneq R \gamma'$. On en déduit que
\[ \ell( (R\gamma' + R \beta) / R \gamma') = \ell( (R\alpha + R \beta) / R\gamma') < \ell( (R\alpha + R \beta) / R\alpha) . \]
Ainsi, $(\gamma' , \beta) < (\alpha , \beta)$. A nouveau par hypothèse de récurrence, il existe $\gamma \in M$ tel que
\[ R \gamma = R \gamma' + R \beta = R\alpha + R \beta = M . \]
Cet élément $\gamma$ engendre donc $M$ en tant que $R$-module.
\end{proof}

\begin{cor}\label{coromono}
Tout $\Dkq$-module holonome est localement monogène.
\end{cor}
\begin{proof}
Soit $U$ un ouvert affine de $\X$ muni d'une coordonnée étale et $\E$ un $\Dkq$-module holonome.
L'algèbre $\Dkq(U)$ est simple par la proposition \ref{prop1.10}. Elle est aussi de longueur infinie à gauche et à droite.
En effet, la suite $\left(\Dkq(U) \cdot (\varpi^k \partial)^n\right)_{n \in \N}$ est strictement décroissante puisque l'algèbre $\Dkq(U)$ est intègre (la norme $\| \cdot \|_k$ est multiplicative).
D'après la proposition \ref{prop2.18}, le module $\E(U)$ est de longueur finie puisque $\E$ est holonome.
Le théorème \ref{stafford} assure alors l'existence d'un élément $e \in \E(U)$ tel que $\E_{| U} \simeq \widehat{\mathcal{D}}^{(0)}_{U, k, \Q} \cdot e$.
\end{proof}

Soit $\E$ un $\Dkq$-module holonome. Il est localement monogène.
Soit $U$ un ouvert de $\X$ sur lequel $\E$ engendré par une section $e$.
Alors $\I = \ann_{\widehat{\mathcal{D}}^{(0)}_{U, k, \Q}} (u)$ est un idéal cohérent non nul de $\widehat{\mathcal{D}}^{(0)}_{U, k, \Q}$.
Autrement l'application $\widehat{\mathcal{D}}^{(0)}_{U, k, \Q}$-linéaire $\widehat{\mathcal{D}}^{(0)}_{U, k, \Q} \to \E_{|U}$, $P \mapsto P \cdot u$ serait injective et $\E_{|U}$ serait aussi de longueur infinie.
Ainsi, $\E_{|U} \simeq \widehat{\mathcal{D}}^{(0)}_{U, k, \Q} / \I$ pour un idéal cohérent $\I$ non nul de $\widehat{\mathcal{D}}^{(0)}_{U, k, \Q}$.
Réciproquement, on a vu que tout $\Dkq$-module cohérent de la forme $\E = \Dkq / \I$, où $\I$ est un idéal non nul, est holonome. On peut maintenant énoncer plusieurs caractérisations des $\Dkq$-modules holonomes.

\begin{prop}\label{prop3.22}
Soit $\E$ un $\Dkq$-module cohérent. Les points suivants sont équivalents :
\begin{enumerate}
\item
$\E$ est holonome  ;
\item
$\E$ est localement de la forme $\widehat{\mathcal{D}}^{(0)}_{U, k, \Q} /\I$ pour un idéal cohérent $\I \neq 0$  ;
\item
$\E$ est de longueur finie  ;
\item
$\E$ est de torsion : pour tout ouvert affine $U$ de $\X$ et pour toute section $m \in \E(U)$, il existe $P \in \Dkq(U)$ non nul tel que $P \cdot m = 0$.
\end{enumerate}
\end{prop}
\begin{proof}
Les deux premiers points sont équivalents. D'après le théorème de Stafford et le corollaire \ref{coromono}, le point 3 est équivalent aux premiers.

On suppose maintenant $\E$ de longueur finie. Soit $U$ un ouvert affine de $\X$ et $(P \mod \I(U))$ un opérateur non nul de $\Dkq(U) / \I(U)$.
Puisque $\E(U)$ est de longueur finie et $\Dkq(U)$ est de longueur infinie, l'application $\Dkq(U) \to \E(U)$, $Q \mapsto Q \cdot (P \mod \I(U))$ n'est pas injective.
Ainsi, l'opérateur $(P \mod \I(U))$ est annulé par un élément non nul de $\Dkq(U)$ et $\E$ est un module de torsion.

Réciproquement, on suppose le module $\E$ de torsion. On se ramène au cas où $\X $ est affine en considérant un recouvrement ouvert affine fini de $\X$.
Comme le module $\E$ est cohérent, $\E$ est engendré par des sections globales $e_1, \dots ,e_r$.
On démontre par récurrence sur $r$ que le module $\E$ est holonome. Si $r = 1$, alors $\E \simeq \Dkq /\I$ où $\I$ est l'idéal annulateur de $e_1$.
Cet idéal est non nul car $e_1$ est de torsion et donc $\E$ est holonome.
Sinon par hypothèse de récurrence, le module $\E' = \Dkq \cdot e_2 + \dots + \Dkq \cdot e_r$ est de longueur finie.
Puisque $\E / \E' = \Dkq \cdot \overline{e_1}$ est aussi de longueur finie, le module $\E$ est forcément de longueur finie.
\end{proof}

On relie maintenant les modules holonomes aux modules à connexion  intégrable.
On identifie $X$ avec la section nulle $s : X \to T^*X$ du fibré cotangent $T^*X$ de $X$.
Le lemme suivant caractérise les modules à connexion intégrable.

\begin{lemma}\label{modconnexion}
Soit $\E$ un $\Dkq$-module holonome. Les énoncés suivants sont équivalents.
\begin{enumerate}
\item
Le $\Dkq$-module $\E$ est localement un $\O_{\X , \Q}$-module libre de rang fini.
\item
La variété caractéristique $\Car(\E)$ de $\E$ est incluse dans $X$.
\item
Le module $\E$ est localement de la forme $\Dkq / P$ avec $P$ un opérateur différentiel fini unitaire d'ordre égale au rang de $\E$ sur $\O_{\X , \Q}$.
\end{enumerate}
\end{lemma}
\begin{proof}
On peut supposer que $X$ est affine muni d'une coordonnée locale.
Dans ce cas, $\gr \Dks \simeq \O_X[\xi]$. On suppose que le module $\E$ est non nul.
Puisque $\E$ est holonome, $\E$ est de la forme $\Dkq / \I$ pour un idéal cohérent non nul $\I$ de $\Dkq$.
Alors $E = \E \otimes_\V \kappa$ est un $\gr \Dks$-module cohérent de la forme  $\O_X[\xi] / I$ pour un certain idéal $I$ non nul.

On suppose tout d'abord que $\E$ est un $\O_{\X , \Q}$-module libre de rang fini $d$.
Il en découle que $E$ est un $\O_X$-module libre de rang $d$. Il existe des sections $e_1 , \dots , e_d$ de $E(X)$ telles que $E = \O_X \cdot e_1 \oplus \dots \oplus \O_X \cdot e_d$.
La famille $\{ \xi^n \cdot e_i \}_{n \in \Z}$ est liée sur $\O_X$. On peut donc trouver un entier $m \geq 1$ et des fonctions $a_j \in \O_X(X)$ tels que
\[ (\xi^m + a_{m-1} \xi^{m-1} + \dots + a_0) \cdot e_i = 0 . \]
Il en découle que la section $e_i$ est annulée par un polynôme unitaire $P_i$ de $\O_X[ \xi]$. Le polynôme unitaire $P = P_1  \dots P_n$ annule tous les éléments $e_1 , \dots , e_n$.
Le polynôme $P$ annule donc le module $E= \O_X \cdot e_1 \oplus \dots \oplus \O_X \cdot e_d$.
On en déduit que $P \in I$ et que $E = \O_X[\xi] / I$ est un $\O_X[\xi]$-module quotient de $\O_X[\xi] / P$.
Il vient $\Car(\E) = \Car(E) \subset \Car(\O_X[\xi] / P)$. Puisque $P$ est unitaire, on a $\Car(\O_X[\xi] / P) = X$.
Ainsi, $\Car(\E) \subset X$.

On suppose maintenant que la variété caractéristique de $\E$ est contenue dans $X$.
Soit $x $ un point de $X$. Quitte à étendre les scalaires $\kappa$ par une extension finie, on peut supposer que $x$ est $\kappa$-rationnel.
L'hypothèse $\Car(E) \subset X$ et la proposition \ref{prop3.10} impliquent que $N_k(\I) = N_k(I) = v_{\O_{X,x}}(I) = 0$.
Toute base de division de $\I$ en $x$ est donc réduite à un unique opérateur différentiel $P$ vérifiant $N_k(P) = 0$.
La condition $N_k(P) = 0$ signifie que le coefficient d'ordre $\Nb(P)$ de $P$ est inversible dans $\O_{\X , \Q}$ au voisinage de $x$.
Un tel opérateur $P$ est défini sur un ouvert affine de $X$ contenant $x$.
Quitte à réduire $X$, on peut supposer que $P \in \Dkq(X)$ et que le coefficient d'ordre $\Nb(P)$ de $P$ est inversible dans $\O_{\X , \Q}(\X)$.
Puisque $P$ est une base de division de l'idéal $\I$, on sait que $P$ engendre l'idéal $\I$. Ainsi, $\E \simeq \Dkq / P$.
Le corollaire \ref{corhensel} dit qu'il est possible de trouver un opérateur différentiel $Q$ unitaire d'ordre $\Nb(P)$ tel que $\E \simeq \Dkq / Q$.
On obtient le troisième point de la proposition.

Enfin le corollaire \ref{corhensel} assure que $\E$ est, localement au voisinage de $x$, un $\O_{\X, \Q}$-module  libre de rang $\Nb(Q)$.
Le schéma formel $\X$ étant irréductible, le nombre $\Nb(Q)$ ne dépend ni de $Q$ ni de $x$ d'après le corollaire \ref{cor1.2.12}.
On note $d$ cet entier. Pour résumer, $\E$ est localement un $\O_{\X , \Q}$-module libre de rang $d$.
\end{proof}

On en déduit une caractérisation des $\Dkq$-modules holonomes via les modules à connexion intégrable.

\begin{cor}
Soit $\E$ un $\Dkq$-module cohérent. Alors $\E$ est holonome si et seulement si il existe un ouvert non vide $U$ de $\X$ tel que $\E_{|U}$ soit un module à connexion intégrable.
\end{cor}
\begin{proof}
On suppose dans un premier temps que le module $\E$ est holonome. La variété caractéristique de $\E$ a un nombre fini de composantes irréductibles verticales.
On note $U$ l'ouvert de $\X$ obtenu en ôtant à $\X$ les abscisses des composantes verticales de $\Car(\E)$.
Par définition de $U$, on a $\Car(\E_{|U}) \subset U$. On en déduit que $\E_{|U}$ est un module à connexion intégrable d'après le lemme \ref{modconnexion}.

Réciproquement, soit $U$ un ouvert non vide de $\X$ pour lequel $\E_{|U}$ est un module à connexion intégrable.
Dans ce cas, $\Car(\E_{|U}) \subset U$, toujours d'après le lemme \ref{modconnexion}.
Si $\E$ n'est pas holonome, alors $\Car(\E) = T^*X$. En effet, $T^*X$ est irréductible puisque $X$ l'est et $\Car(\E)$ est une sous-variété fermée de $T^*X$ de dimension maximale deux.
En particulier, on aurait $\Car(\E_{|U}) = T^*U$. Cela contredit l'hypothèse $\Car(\E_{|U}) \subset U$.
Ainsi $\E$ est holonome.
\end{proof}

\subsection{Caractérisation cohomologique des modules holonomes}\label{partie3.5}

On énonce tout d'abord plusieurs résultats démontrés par Anne Virrion dans l'article \cite{virrion} pour les $\widehat{\mathcal{D}}^{(0)}_{\X, \Q}$-modules cohérents.
Les preuves et les propositions se généralisent sans difficulté pour un niveau de congruence $k$ quelconque.
En effet, les arguments des preuves des énoncés de Virrion utilisés dans ici se démontrent au niveau du gradué $\gr \widehat{\mathcal{D}}^{(0)}_{\X}$ avant de remonter à $\widehat{\mathcal{D}}^{(0)}_{\X}$ de manière classique.
Il suffit donc de vérifier les mêmes propriétés pour le gradué $\gr \widehat{\mathcal{D}}^{(0)}_{\X , k}$, ce qui est clair puisque $\gr \widehat{\mathcal{D}}^{(0)}_{\X , k} \simeq \gr \widehat{\mathcal{D}}^{(0)}_{\X }$.
On démontre ensuite qu'un $\Dkq$-module cohérent $\M$ est holonome si et seulement si
\[ \forall d \neq 1, ~~ \Ext_{\Dkq}^d (\M , \Dkq) = 0 . \]
Enfin, on définit un foncteur dualité de la catégorie des $\Dkq$-modules holonomes dans elle même vérifiant un isomorphisme de bidualité.

\vspace{0.4cm}

La proposition suivante se démontre comme le théorème 4.3 du chapitre 0 de \cite{virrion} dans le cas où $\X$ est une courbe formelle.
La preuve repose sur le calcul de la dimension du gradué du faisceau $\widehat{\mathcal{D}}^{(0)}_{\X}$ qui est identique au gradué du faisceau $\widehat{\mathcal{D}}^{(0)}_{\X  , k}$.
Les conclusions restent donc valides pour un niveau de congruence $k$ quelconque.

\begin{prop}
La dimension cohomologique du faisceau $\Dk$ est égale à trois et la dimension cohomologique du faisceau $\Dkq$ est inférieure ou égale à trois.
\end{prop}

Soit $\M$ un $\Dkq$-module cohérent. On pose
\[ \dim \M = \dim (\Car(\M)) \in \{ 0 , 1 , 2 \}, ~~  \codim \M = 2 \dim X - \dim \M = 2 - \dim \M . \]

L'inégalité de Bernstein se traduit de la manière suivante sur la codimension : $\M \neq 0$ si et seulement si $\codim \M \leq 1$.
Par ailleurs, $\M$ est un module holonome si et seulement si $\codim \M = 1$.
On note
\[ \omega_{\X , \Q} := \left( \wedge_{i=0}^1 \Omega_\X^1 \right) \otimes_\V K .\]
C'est un $\O_{\X , \Q}$-module libre de rang un.
La proposition 2.1.1 du chapitre 1 de \cite{virrion} appliquée au faisceau $\Dkq$ implique le résultat suivant.

\begin{prop}
Le foncteur $\bullet  \otimes_{\O_{\X , \Q}} \omega_{\X , \Q}^{-1}$ induit une équivalence de catégorie entre la catégorie des $\Dkq$-modules à droites et la catégorie des $\Dkq$-modules à gauche.
\end{prop}

On note $D_c^b(\Dkq)$ la catégorie dérivée formée des complexes bornés de $\Dkq$-modules cohérents.
On identifie la catégorie des $\Dkq$-modules cohérents avec les complexes concentrés en degré zéro.
Pour tout complexe $\M$ de $D_c^b(\Dkq)$, on définit son dual $\Du (\M)$ par
\[ \Du(\M) := \mathcal{R H}om (\M , \Dkq[1]) \otimes_{\O_{\X , \Q}} \omega_{\X , \Q}^{-1} .\]
Virrion démontre en toute généralité dans le chapitre trois de \cite{virrion} que $\Du$ est un foncteur de la catégorie $D_c^b(\Dkq)$ dans elle même et que pour tout complexe $\M$ de $D_c^b(\Dkq)$, il existe un isomorphisme canonique $\M \simeq \Du \circ \Du (\M)$.

\vspace{0.4cm}

On rassemble dans la proposition suivante le corollaire 2.3 et la proposition 3.5 du chapitre 3 de \cite{virrion}.
Ces résultats se démontrent tout d'abord au niveau du gradé $\gr \widehat{\mathcal{D}}^{(0)}_{\X} = \gr \widehat{\mathcal{D}}^{(0)}_{\X , k}$.
Le passage du gradué au faisceau $\widehat{\mathcal{D}}^{(0)}_{\X}$ est systématique et valide plus généralement.
En particulier, cela fonctionne aussi pour le faisceau $\Dk$.

\begin{prop}\label{prop3.27}
Soit $\M$ un $\Dkq$-module cohérent non nul. Alors
\begin{enumerate}
\item
$\forall i \geq 0$, $\codim( \Ext^i_{\Dkq}(\M , \Dkq) ) \geq i$ ;
\item
$\codim \M = \inf \left\{ i \in \N : \Ext^i_{\Dkq} (\M , \Dkq) \neq 0 \right\}$.
\end{enumerate}
\end{prop}

On peut maintenant démontrer la caractérisation cohomologique suivante des $\Dkq$-modules holonomes.

\begin{prop}\label{prop3.28}
Soit $\M$ un $\Dkq$-module cohérent. Alors $\M$ est holonome si et seulement si
\[ \forall i \neq 1, ~~ \Ext_{\Dkq}^i(\M , \Dkq) = 0 .\]
De plus, si le module $\M$ est holonome, alors $\M^* :=\Ext_{\Dkq}^1(\M , \Dkq) \otimes_{\O_{\X , \Q}} \omega_{\X , \Q}^{-1}$ est aussi un $\Dkq$-module holonome.
\end{prop}
\begin{proof}
Soit $\M$ un $\Dkq$-module cohérent que l'on peut supposer non nul. Par construction,
$\Ext^i_{\Dkq}(\M , \Dkq) \otimes_{\O_{\X , \Q}} \omega_{\X , \Q}^{-1}$ est un $\Dkq$-module à gauche cohérent. Il vérifie donc l'inégalité de Bernstein.
Autrement dit, $\codim( \Ext^i_{\Dkq}(\M , \Dkq) ) \leq 1$ ou $\Ext^i_{\Dkq}(\M , \Dkq) = 0$.
Par ailleurs, on sait que $\codim( \Ext^i_{\Dkq}(\M , \Dkq) ) \geq i$ d'après la proposition \ref{prop3.27}.
Ainsi $\Ext^i_{\Dkq}(\M , \Dkq)  \neq 0$ implique $i \leq 1$. On a donc toujours $\Ext^i_{\Dkq}(\M , \Dkq) = 0$ dès que $i \geq 2$.

On suppose maintenant que le module $\M$ est holonome. Alors $\dim \M = \codim \M = 1$.
Le second point de la proposition \ref{prop3.27} implique que $\Ext^i_{\Dkq}(\M , \Dkq) = 0$ pour tout entier $i \neq 1 = \codim \M $.
Réciproquement, on suppose que $\Ext^i_{\Dkq}(\M , \Dkq) = 0$ dès que $i \neq 1$.
Le second point de la proposition implique que $\codim(\M) = 1$. Autrement dit, $\M$ est un module holonome.

Il reste à montrer que le module $\M^* = \Ext_{\Dkq}^1(\M , \Dkq) \otimes_{\O_{\X , \Q}} \omega_{\X , \Q}^{-1}$ est holonome dès que $\M$ est holonome.
On sait d'après la proposition \ref{prop3.27} que $\codim \M^* \geq 1$.
Si $\codim \M^* = 2$, alors $\M^* = 0$ d'après l'inégalité de Bernstein.
Cela contredit l'hypothèse $\Ext_{\Dkq}^1(\M , \Dkq) \neq 0$. Donc $\codim \M^* = 1$ et $\M^*$ est holonome.
\end{proof}

\begin{remark}
On a démontré que $\Ext^i_{\Dkq}(\M , \Dkq) = 0$ dès que $i \geq 2$ et $\M$ est un $\Dkq$-module cohérent.
\end{remark}

On démontre enfin que le dual d'un $\Dkq$-module holonome demeure holonome.

\begin{cor}
Le foncteur dualité préserve la catégorie des modules holonomes.
De plus, si $\M$ est un $\Dkq$-module holonome, alors $\Du (\M) \simeq \M^* = \Ext_{\Dkq}^1(\M , \Dkq) \otimes_{\O_{\X , \Q}} \omega_{\X , \Q}^{-1}$.
\end{cor}
\begin{proof}
Soit $\M$ un $\Dkq$-module holonome. On note $\M^*$ le $\Dkq$-module holonome $\Ext_{\Dkq}^1(\M , \Dkq) \otimes_{\O_{\X , \Q}} \omega_{\X , \Q}^{-1}$.
On sait d'après la proposition \ref{prop3.28} que pour tout $i \neq 1$, $\Ext_{\Dkq}^i(\M , \Dkq) = 0$. On a donc $\mathcal{H}^i(\Du(\M)) = 0$ pour tout entier $i \neq 0$.
On en déduit que $\Du(\M) \simeq \mathcal{H}^0(\Du(\M)) \simeq \M^*$ est un $\Dkq$-module holonome.
L'isomorphisme de bidualité $\M \simeq (\M^*)^*$ provient du théorème 3.6 du chapitre 1 de \cite{virrion}.
\end{proof}

\section{$\Di$-modules coadmissibles}\label{section4}

On introduit dans cette dernière section une catégorie abélienne formée de $\Di$-modules coadmissibles de longueur finie.
Idéalement, on aimerait définir une catégorie de $\Di$-modules coadmissibles holonomes qui soit une sous-catégorie pleine de celle-ci.
On commence par rappeler les définitions du faisceau $\Di$ et des $\Di$-modules coadmissibles.

\subsection{Définition}\label{partiedefmodulecoad}

Pour plus de détails sur le faisceau $\Di$ et sur les propriétés des $\Di$-modules coadmissibles, le lecteur peut regarder l'article  \cite{huyghe} de Christine Huyghe, Tobias Schmidt et Matthias Strauch.

\vspace{0.4cm}

Soit $U$ un ouvert affine contenant le point $x$ sur lequel on dispose d'une coordonnée locale associée à $x$.
Pour tout entier $k$, l'algèbre $\widehat{\mathcal{D}}^{(0)}_{\X, k + 1 , \Q}(U)$ est une sous-algèbre de $\Dkq(U)$.
On considère les morphismes de transition $\widehat{\mathcal{D}}^{(0)}_{\X, k+1 , \Q} \to \Dkq$ induits par ces inclusions locales.
On définit le faisceau $\Di$ comme la limite projective des faisceaux $\Dkq$.
\begin{definition}
On note $\Di := \varprojlim_k \Dkq$.
\end{definition}

Le faisceau $\Di$ est un faisceau de $K$-algèbres sur le schéma formel $\X$. Il vérifie les trois points suivant :
\vspace{0.1cm}
\begin{enumerate}
\item
$\Di(U)$ est une $K$-algèbre de Fréchet-Stein et sa topologie est induite par les normes $\| \cdot \|_k$ des algèbres de Banach $\Dkq(U)$;
\item
$\Di(U) = \varprojlim_k \Dkq(U) = \bigcap_{k \geq 0} \Dkq (U)$ ;
\item
$ \Di(U) =\displaystyle \left\{ \sum_{n = 0}^\infty a_n \cdot \partial^n :  \,  a_n \in \O_{\X, \Q} (U)~~ \mathrm{tq} ~~ \forall \eta>0,~  \lim_{n \to \infty} a_n \cdot \eta^n =0 \right\}$.
\end{enumerate}

Le lemme suivant caractérise les opérateurs différentiels finis de $\Di(U)$ à l'aide des fonctions $\Nb$.
On en déduit les éléments inversibles de $\Di(U)$.

\begin{lemma}\label{lemmedeg}
Soit $P$ un opérateur différentiel de $\Di(U)$. La suite $(\Nb(P))_{k \geq 0}$ est croissante.
De plus, $P$ est un opérateur fini de degré $d$ si et seulement si la suite $(\Nb(P))_k$ est stationnaire de valeur limite $d$.
\end{lemma}
\begin{proof}
On écrit $P = \sum_{n=0}^\infty a_n \cdot \partial^n$. On commence par montrer que la suite $(\Nb(P))_k$ est croissante.
Les coefficients de $P$ dans la base $(\varpi^k \partial)^n$ de $\Dkq$ sont $\varpi^{-kn} a_n$. Donc
\[ \Nb(P) = \max\{ n \in \N : |\varpi|^{-kn} \cdot |a_n|= \|P \|_k \} .\]
Soit $n_0 = \Na_{k+1}(P)$. Puisque $\|P \|_{k+1} = | \o\varpi^{-(k+1)n_0} \cdot a_{n_0} | > | \varpi^{-(k+1)n} \cdot a_n | $ par définition de $n_0$, on a
\[\forall n > n_0, ~~  |\varpi^{-kn} a_n | = |\varpi|^n \cdot | \varpi^{-(k+1)n} a_n | < |\varpi|^{n_0}\cdot  | \varpi^{-(k+1)n_0} a_{n_0} | =  | \varpi^{-kn_0} a_{n_0} | .\]
On en déduit que $\Nb(P) \leq n_0 = \overline{N}_{k+1}(P)$.

On suppose maintenant que $\Nb(P) = m$ à partir d'un rang $k_0$. Cela signifie que pour tout entier $k \geq k_0$ et pour tout entier $n > m$, $|a_n| \cdot |\varpi|^{-kn} < |a_m| \cdot |\varpi|^{-km}$.
Autrement dit,$|a_n| < |a_m| \cdot |\varpi|^{k(n-m)}$. Mais $|\varpi|^{k(n-m)} \to 0$ lorsque $k\to \infty$.
Le passage à la limite $k \to \infty$ donne $|a_n| = 0$. Ainsi, $P$ est un opérateur fini d'ordre $m$.
\end{proof}

\begin{cor}
Les opérateurs différentiels inversibles de l'algèbre $\Di(U)$ sont les fonctions inversibles :
\[ \Di(U)^\times = \O_{\X,\Q}(U)^\times .\]
\end{cor}
\begin{proof}
Soit $P \in \Di(U)$ un opérateur inversible. Alors $P$ est inversible dans $\Dkq(U)$ pour tout niveau de congruence $k \in \N$.
De manière équivalente, $\Nb(P) =0$ pour tout entier $k$ et le coefficient constant de $P$ inversible d'après le corollaire \ref{corinv}.
Le lemme \ref{lemmedeg} implique que $P$ est un opérateur fini d'ordre 0.
Autrement dit, $P$ est un élément inversible de $\O_{\X,\Q}(U)$.
\end{proof}

On termine cette partie par la définition des $\Di$-modules coadmissibles suivie d'un exemple.

\begin{definition}
Un module coadmissible est un $\Di$-module $\M$  isomorphe à une limite projective $\varprojlim_k \M_k$ de $\Dkq$-modules cohérents $\M_k$ tels que les applications de transition $\M_{k+1} \to \M_k$ soient $\widehat{\mathcal{D}}^{(0)}_{\X, k+1 , \Q}$-linéaires et tels que pour chaque indice $k$, on dispose d'un $\Dkq$-isomorphisme $\Dkq \otimes_{\widehat{\mathcal{D}}^{(0)}_{\X, k+1 , \Q}} \M_{k+1} \simeq \M_k$ induit par l'application de transition.
\end{definition}

La catégorie des $\Di$-modules coadmissibles est abélienne et contient les $\Di$-modules cohérents.
En effet, une présentation finie locale d'un $\Di$-module cohérent $\M$ fournie des présentations finies locales des modules $\M_k := \Dk \otimes_{\Di} \M$.
On dispose de morphismes de transition naturels commutant entre ces présentations finies pour les différents niveaux de congruence $k$.
 On en déduit que $\M$ est bien la limite projective des $\Dkq$-modules cohérents $\M_k$.

Soit $\M = \varprojlim_k \M_k$ un $\Di$-module coadmissible.
Il est démontré dans \cite{huyghe}, théorème 3.1.17 et \cite{schneider}, corollaire 3.1, que $\M_k \simeq \Dkq \otimes_{\Di} \M$ et que $\M \simeq \varprojlim_k \left(\Dkq \otimes_{\Di} \M\right)$ en tant que $\Di$-module coadmissible.
On peut donc choisir $\M_k$ égale à $\Dkq \otimes_{\Di} \M$.

\vspace{0.4cm}

On explicite dans l'exemple ci-dessous un opérateur différentiel infini $P$ de $\Di(U)$ vérifiant $\Nb(P) = k$.
On montre que $\Di / P$ est un $\Di$-module coadmissible isomorphe à une limite projective de la forme $\varprojlim_k \Dkq / P_k$ avec $P_k$ un opérateur différentiel fini d'ordre $k$ de l'algèbre $\Dkq(U)$.

\begin{example}\label{ex4.2}
Soit $P = \prod_{n \geq 1} (1 - \varpi^n \partial) \in \Di(U)$. Alors $\Nb(P) = k$. En effet, le coefficient de $\partial^n$ est à un signe près
\[ \varpi^{1+2+ \dots + n} \cdot (1 + \varpi + \varpi^2 + \dots) + \varpi^{2+ 3 + \dots + (n+1)} \cdot (1 + \varpi + \varpi^2 + \dots) + \dots = \varpi^\frac{n(n+1)}{2} \cdot a_n\]
avec $a_n$ un élément de $\V$ de valeur absolue 1. Dans $\Dkq(U)$, le coefficient d'ordre $n$ de $P$ est $\pm \varpi^{n \left(\frac{n+1}{2}-k \right)} \cdot a_n$.
Par définition, $\Nb(P)$ est le plus grand entier $n$ maximisant la valeur absolue $|\varpi|^{n \left(\frac{n+1}{2}-k \right)}$.
On cherche donc le plus grand entier $n$ minimisant la puissance $n \left(\frac{n+1}{2}-k \right)$.

La fonction $x \mapsto x\left(\frac{x+1}{2} - k\right)$ est minimale en $x = k - \frac{1}{2}$ de valeur $-k^2 -k + \frac{3}{4}$.
La puissance $n \left(\frac{n+1}{2}-k \right)$ est donc minimale pour $n = k-1$ et $n = k$.
Ceci prouve que $\Nb(P) = k$. On a de plus
\[ \| P_k \|_k = |\varpi|^{k \left(\frac{k+1}{2}-k \right)} = | \varpi|^{ -k^2 / 2 + k/2} .\]

Dans $\Dkq(U)$, $P$ s'écrit $P = P_k \cdot Q_k$ avec $P_k = \prod_{ 1 \leq n \leq k} (1 - \varpi^n \partial)$ un opérateur fini d'ordre $\Nb(P_k) = k$ et $Q_k = \prod_{n > k} (1 - \varpi^n \partial)$ inversible dans $\Dkq(U)$ puisque $\Nb(Q_k) = 0$ et puisque son coefficient constant est inversible.
On en déduit que $\Dkq / P \simeq \Dkq / P_k$.
Par ailleurs, $P_{k+1} = (1 - \varpi^{k+1} \partial) \cdot P_k$ avec $1 - \varpi^{k+1} \partial$ inversible dans $\Dkq(U)$. On a donc $\Dkq / P_{k+1} \simeq \Dkq / P_k$.
Autrement dit, $\Di / P \simeq \varprojlim_k \Dkq / P_k$ en tant que $\Di$-module coadmissible.

On peut retrouver $P$ à partir des opérateurs $P_k$ : la suite $(P_k)_k$ converge vers $P$ dans l'algèbre de Fréchet-Stein $\Di(U)$.
En effet, il suffit de le vérifier pour toutes les normes $\| \cdot \|_m$. On a $P - P_k = ( Q_k - 1) \cdot P_k$.
Pour $k \geq m$, on observe que $\overline{N}_m(P_k) = m$ et que $\| P_k \|_m = | \varpi |^{-\frac{1}{2}(m^2-m)}$.
Le coefficient constant de $Q_k - 1$ est nul et le coefficient de $\partial^n$ est de la forme $\varpi^{k + (k+1) + \dots + (k+n-1)} \cdot a_n$, où $a_n \in \V$ est de valeur absolue 1.
Dans $\widehat{\mathcal{D}}^{(0)}_{\X, m, \Q}(U)$, le coefficient de $(\varpi^m \partial)^n$ est $\omega^{n(k+\frac{n-1}{2}-m)} \cdot a_n$.
Les coefficients de $Q_k -1$ sont presque ceux de $P_k$ : il suffit de remplacer $k$ par $m+1-k$.
La fonction $x(k+\frac{x-1}{2}-m)$ est minimale pour $x = m+1-k - 1/2$. 
Pour $k$ assez grand, par exemple $k \geq m+1$, ce terme est négatif. La norme de $1-Q_k$ est donc donnée par le coefficient d'indice un : $\Nb(Q_k -1) = 1$ et $ \| Q_k -1 \|_m = |\varpi|^{k-m}$. Il vient
\[ \| P - P_k \|_m = \|1 - Q_k \|_m \cdot \| P_k\|_m = | \varpi|^{k-\frac{m^2}{2} - \frac{m}{2}} \underset{k \to \infty}{\to} 0.\]
\end{example}

\subsection{Une catégorie de $\Di$-modules coadmissibles de longueur finie}

Soit $\M_k$ un $\Dkq$-module holonome.
On note $m(\M_k)$ la longueur du cycle caractéristique de $\M_k$. Si $I(\Car(\M_k))$ est l'ensemble des composantes irréductibles de la variété caractéristique de $\M_k$, alors
\[ m(\M_k) := \sum_{C \in I(\Car(\M_k))} m_C(\M_k) \in \N .\]
Cette longueur est un entier naturel puisque les multiplicités $m_C(\M_k)$ le sont. Le corollaire \ref{cor3.14} implique que $\M_k = 0$ si et seulement si $m(\M_k) = 0$.

\vspace{0.4cm}

Soit maintenant $\M = \varprojlim_k \M_k$ un $\Di$-module coadmissible.
On suppose qu'il existe un niveau de congruence $k_0 \in \N$ tel que $\M_k$ soit un $\Dkq$-module holonome pour tout entier $k \geq k_0$.
On note $k_\M$ le plus petit entier naturel pour lequel $\M_k$ est holonome dès que $k \geq k_\M$.
On associe à un tel module coadmissible $\M$ une multiplicité $m(\M)$ définie par
\[ m(\M) = \limsup_{k \geq k_\M} \{ m(\M_k) \} =  \inf_{k \geq k_\M}  \left\{ \sup_{k' \geq k} m(\M_{k'}) \right\} \in \N \cup \{ \infty \} . \]

\begin{definition}\label{def4.4}
On note $\H$ la catégorie constituée des $\Di$-modules coadmissibles $\M= \varprojlim_k \M_k$ vérifiant les deux points suivant.
\begin{enumerate}
\item
Il existe un niveau de congruence à partir duquel les $\Dkq$-modules cohérents $\M_k$ sont tous holonomes.
\item
La multiplicité $m(\M)$ de $\M$ est finie, autrement dit $m(\M) \in \N$.
\end{enumerate}

\end{definition}

C'est une sous-catégorie pleine de la catégorie abélienne des $\Di$-modules coadmissibles.
On démontre dans ce qui suit que la catégorie $\H$ est abélienne et que tout object $\M$ de $\H$ est de longueur finie inférieure ou égale à $m(\M)$.

\vspace{0.4cm}

Soit $\M = \varprojlim_k \M_k$ un object de $\H$. Par définition, $m(\M) < \infty$.
Autrement dit, il existe un entier $k_0 \geq k_\M$ pour lequel $ \sup_{k \geq k_0} \{ m(\M_{k}) \} < \infty$.
La suite $( \sup_{k' \geq k} \{ m(\M_{k'}) \} )_{k \geq k_0}$ est décroissante formée d'entiers naturels.
Elle est donc stationnaire. Sa valeur limite est exactement $m(\M)$.
On en déduit qu'il existe un niveau de congruence $k_1 \geq k_\M$ pour lequel
\[ \forall k \geq k_1 , ~~ m(\M) = \sup_{k' \geq k} \{ m(\M_{k'}) \} . \]

Soit $\xymatrix{ 0 \ar[r] & \M \ar[r] & \Nn \ar[r] & \L \ar[r] & 0 }$ une suite exacte de $\Di$-modules coadmissibles.
On écrit $\M = \varprojlim_k \M_k$, $\Nn = \varprojlim_k \Nn_k$ et $\L = \varprojlim_k \L_k$.
Pour tout entier naturel $k$, cette suite induit une suite exacte de $\Dkq$-modules cohérents :
\[ \xymatrix{ 0 \ar[r] & \M_k \ar[r] & \Nn_k \ar[r] & \L_k \ar[r] & 0 } .\]
On note $k_0 = \max\{ k_\M, k_\Nn, k_\L \} \in \N$. Pour tout entier $k \geq k_0$, les modules $\Nn_k$, $\M_k$ et $\L_k$ sont holonomes par définition de $\H$.
Pour $k \geq k_0$, on sait d'après la proposition \ref{prop3.17} que $\cc(\Nn_k) = \cc(\M_k) + \cc(\L_k)$. Il en découle que pour tout $k \geq k_0$, 
\[ m(\Nn_k) = m(\M_k) + m(\L_k) .\]
On en déduit immédiatement la proposition suivante.

\begin{prop}\label{propsuiteexacte}
Soit $\xymatrix{ 0 \ar[r] & \M \ar[r] & \Nn \ar[r] & \L \ar[r] & 0 }$ une suite exacte de $\Di$-modules coadmissibles vérifiant le premier point de la définition \ref{def4.4}. Alors
\begin{enumerate}
\item
$m(\M) \leq m(\Nn)$ et $m(\L) \leq m(\Nn)$ ;
\item
$m(\Nn) \leq m(\M) + m(\L)$.
\end{enumerate}
En particulier, $\Nn \in \H$ si et seulement si $\M \in \H$ et $\L \in \H$.
\end{prop}
\begin{proof}
On note $\M = \varprojlim_k \M_k$, $\Nn = \varprojlim_k \Nn_k$ et $\L = \varprojlim_k \L_k$.
Il existe un niveau de congruence $k_1 \geq \max\{ k_\M, k_\Nn, k_\L \}$ pour lequel $m(\M) = \sup_{k '\geq k} \{ m(\M_{k'}) \}$, $m(\Nn) = \sup_{k' \geq k} \{ m(\Nn_{k'}) \}$ et $m(\L) = \sup_{k' \geq k} \{ m(\L_{k'}) \}$ dès que $k \geq k_1$.

Pour tout entier naturel $k \geq k_1$, on a $m(\Nn_k) = m(\M_k) + m(\L_k)$ d'après la proposition \ref{prop3.17}.
Les inégalités $m(\Nn_k) \geq m(\M_k)$ pour $k \geq k_1$ donnent $m(\Nn) \geq m(\M)$ en passant à la borne supérieure sur $k \geq k_1$.
De même, $m(\Nn_k) \geq m(\L_k) $. Enfin, les inégalités
\[ m(\Nn_k) = m(\M_k) + m(\L_k) \leq \sup_{k \geq k_1} \{ m(\M_{k}) \} + \sup_{k \geq k_1} \{ \mu(\M_{k}) \} = m(\M) + m(\L) \]
pour tout $k \geq k_1$ impliquent que $m(\Nn) = \sup_{k \geq k_1} \{ m(\Nn_{k}) \} \leq m(\M) + m(\L)$.
\end{proof}

\begin{remark}
Bien que $m(\Nn_k) = m(\M_k) + m(\L_k)$ pour un niveau de congruence $k$ fixé, la multiplicité $m$ n'est a priori pas additive pour les modules coadmissibles.
En effet, ces égalités deviennent seulement des inégalités en passant à la borne supérieure.
\end{remark}

Cette proposition montre que la catégorie $\H$ est abélienne.
L'exemple suivant assure qu'elle n'est pas triviale : elle contient les $\Di$-modules coadmissibles de la forme $\Di /P$ pour $P$ un opérateur différentiel fini de $\Di$.

\begin{example}\label{examplelf}
On suppose que $\X = U$ est affine. Soit $P \in \Di(\X)$.
On considère le $\Di$-module coadmissible $\M = \Di / P = \varprojlim_k \M_k$ avec $\M_k = \Dkq / P$.
Les $\Dkq$-modules cohérents $\M_k$ sont tous holonomes d'après la proposition \ref{prop3.22}.
\begin{enumerate}
\item
On regarde tout d'abord ce qu'il se passe lorsque $P$ est un opérateur infini. La suite $(\Nb(P))$ est croissante et diverge vers $+\infty$ d'après le lemme \ref{lemmedeg}.
La proposition \ref{prop3.10} implique que $m(\M_k) \geq \Nb(P)$.
On en déduit que $m(\M) = +\infty$. Donc $\M$ n'est pas un object de la catégorie $\H$.
\item
On suppose maintenant que $P = \sum_{n = 0}^d a_n \cdot \partial^n$ est un opérateur fini d'ordre $d$.
Alors $\Nb(P) = d$ pour $k$ assez grand d'après le lemme \ref{lemmedeg}.
On ne considère maintenant que ces indices $k$.
On sait d'après la proposition \ref{prop3.10} que les multiplicités de $\Car(\M_k)$ en $x$ sont les nombres $\Nb(P) = d$ et $N_k(P , x) = N(a_d, x) =$ valuation de $(a_d \mod \varpi)$ dans $\O_{ X , x}$.
En particulier, $x$ est l'abscisse d'une composante irréductible verticale de $\Car \M_k$ si et seulement si $N_k(P , x) > 0$.
La multiplicité de cette composante est alors $N_k(P , x)$.
Si $x_1 , \dots , x_s$ sont les zéros de $a_d$, alors pour $k$ suffisamment grand,
\[ m(\M_k) = d + N(a_d , x_1) + \dots + N(a_d , x_s) .\]
Ces multiplicités ne dépendent plus de $k$. On en déduit que
\[ m(\M) = \limsup_{k \geq 0} \{ m(\M_k) \} = d + N(a_d , x_1) + \dots + N(a_d , x_s) < \infty .\]
Autrement dit, $\M = \Di / P$ appartient à la catégorie $\H$.
\end{enumerate}
\end{example}

Le lemme suivant montre que la multiplicité $m$ caractérise les $\Di$-modules holonomes nuls.
Cela provient du fait qu'un $\Dkq$-module holonome $\M_k$ est nul si et seulement si $m(\M_k) = 0$.

\begin{lemma}\label{lemmezero}
Un élément $\M$ de $\H$ est nul si et seulement si $m(\M) = 0$.
\end{lemma}
\begin{proof}
On écrit $\M = \varprojlim_k \M_k$. Si $\M = 0$, alors $\M_k = 0$.
Les multiplicités de $\M_k$ sont toutes nulles par définition et $m(\M_k) = 0$. Alors $m(\M) = 0$.

On suppose maintenant que $m(\M) = 0$. Par définition de $m(\M)$, il existe un niveau de congruence $k$ à partir duquel $m(\M_k) = 0$.
Autrement dit, les multiplicités du module $\M_k$ sont toutes nulles et $\M_k = 0$ d'après le corollaire \ref{cor3.14}.
Ainsi, $\M_k = 0$ pour $k$ suffisamment grand et donc $\M = 0$.
\end{proof}

Bien que la multiplicité $m$ ne soit pas additive sur les suites exactes, on peut démontrer que les éléments de $\H$ sont de longueur finie en utilisant la proposition \ref{propsuiteexacte} et le lemme \ref{lemmezero}.

\begin{prop}\label{propfinitelenght}
Soit $\M$ un élément de $\H$. Alors $\M$ est de longueur finie inférieure ou égale à $\mu(\M)$.
\end{prop}
\begin{proof}
Soit $\M = \varprojlim_k \M_k$ un élément de $\H$. On démontre que toute suite décroissante $(\M^n)_{n \in \N}$ de sous-$\Di$-modules coadmissibles de $\M$ est stationnaire.
On peut supposer que $\M^0 = \M$. Comme $\M^n$ est un sous-module de $\M$, on sait que $m(\M^n) \leq m(\M)$ d'après la proposition \ref{propsuiteexacte}.
La suite $(m(\M^n))_n$ est une suite décroissante d'entiers naturels. Elle est donc stationnaire. Il existe un entier naturel $n_0$ tel que pour tout $n \geq n_0$, $m(\M^{n+1}) = m(\M^n)$.
On suppose dans la suite que $n \geq n_0$. On écrit $\M^n = \varprojlim_k \M^n_k$.
Il existe un rang $k_n \geq \max\{k_{\M^n} , k_{\M^{n+1}} \}$ pour lequel
\begin{equation}\label{eqegalite}
\forall k \geq k_n, ~~ m(\M^n)= \sup_{k' \geq k} \{ m(\M^n_{k'}) \}= m(\M^{n+1})= \sup_{k' \geq k} \{ m(\M^{n+1}_{k'}) \}
\end{equation}
Pour tout $k \geq k_n$, on considère la suite exacte de $\Dkq$-modules cohérents
\[ \xymatrix{ 0 \ar[r] & \M_k^{n+1} \ar[r] & \M^n_k \ar[r] & \M^n_k / \M_k^{n+1} \ar[r] & 0 }. \]
On sait que $m(\M^n_k) = m(\M_k^{n+1}) + m(\M^n_k / \M_k^{n+1})$ d'après la proposition \ref{prop3.17}.
En particulier, si $m(\M^{n+1})= m(\M^{n+1}_{k})$ pour un certain niveau de congruence $k \geq k_n$, alors $m(\M^n)= m(\M^n_{k})$ d'après l'égalité (\ref{eqegalite}) puisque $m(\M^n_k) \geq m(\M_k^{n+1})$.
On en déduit l'égalité des multiplicités $m(\M^{n+1}_{k}) = m(\M^n_k)$ et que $m(\M^n_k / \M_k^{n+1}) = 0$.
Le corollaire \ref{cor3.14} implique alors que $\M^n_k / \M_k^{n+1} = 0$. Autrement dit, $\M_k^n \simeq \M_k^{n+1}$.
Puisque l'égalité \ref{eqegalite} est vérifiée pour tout entier $k \geq k_n$, il existe une infinité d'entiers $k \geq k_n$ pour lesquels $\M_k^n \simeq \M_k^{n+1}$.

Alors $\M^n \simeq \M^{n+1}$ en tant que $\Di$-modules coadmissibles.
En effet, soit $(k_\ell)_{\ell \geq 0}$ une suite strictement croissante d'entiers naturels telle que pour tout $\ell \in \N$, $\M_{k_\ell}^n = \M_{k_\ell}^{n+1}$.
La propriété universelle de la limite projective permet d'obtenir des isomorphismes de $\Di$-modules $\M^n \simeq \varprojlim_\ell \M_{k_\ell}^n$ et $\M^{n+1} \simeq \varprojlim_\ell \M_{k_\ell}^{n+1}$.
Puisque $\M^{n+1}$ est un sous-module de $\M^n$, les morphismes de transition des modules $\M^{n+1}_k$ sont induits par ceux des $\M^n_k$.
On en déduit que les morphismes de transition $\M_{k_{\ell+1}}^n = \M_{k_{\ell+1}}^{n+1} \to \M_{k_\ell}^n = \M_{k_\ell}^{n+1}$ des modules $\M_{k_\ell}^n$ sont aussi les morphismes de transition des modules $\M_{k_\ell}^{n+1}$.
Il en découle, en passant à la limite projective sur $\ell$, que $\M^n \simeq \M^{n+1}$.

On a démontré que pour tout entier $n \geq n_0$, $\M^n \simeq \M^{n+1}$. La suite $(\M^n)_n$ est donc stationnaire.
On a aussi démontré que $m(\M^{n+1}) = m(\M^n)$ implique $\M^{n+1} \simeq \M^n$ lorsque $\M^{n+1}$ est un sous-module de $\M^n$.
Comme la suite $(m(\M^n))_n$ est décroissante de terme initial $\mu(\M) < \infty$, la longueur d'une suite strictement décroissante de sous-modules de $\M$ est de longueur au plus $m(\M)$.
De même, toute suite strictement croissante de sous-modules de $\M$ est de longueur au plus $m(\M)$.
Ainsi, $\M$ est un $\Di$-module de longueur finie inférieure ou égale à $m(\M)$.
\end{proof}

\begin{example}
On continue l'exemple \ref{examplelf}.
On suppose toujours que $\X = U$ est affine muni d'une coordonnée locale.
Soit $P = \sum_{n = 0}^d a_n \cdot \partial^n$ un opérateur fini d'ordre $d$ de $\Di(\X)$.
On note $x_1 , \dots , x_s$ les zéros de $a_d$. On rappelle que $N(a_d, x)$ est la valuation de $(a_d \mod \varpi)$ dans l'anneau de valuation discrète $\O_{ X , x}$.
Le $\Di$-module coadmissible $\Di / P$ est de longueur finie inférieure ou égale à $m(\M) = d + N(a_d , x_1) + \dots + N(a_d , x_s)$ d'après la proposition.
\end{example}

On termine cette partie en démontrant que tout $\Di$-module à connexion intégrable est de longueur finie.

\begin{lemma}
Soit $\M = \varprojlim_k \M_k$ un $\Di$-module coadmissible tel que les $\M_k$ soient des $\O_{\X , \Q}$-modules libres de rang $n$ pour $k$ suffisamment grand.
Alors $\M$ est un $\O_{\X , \Q}$-module libre de rang $n$.
\end{lemma}
\begin{proof}
Par hypothèse, il existe un niveau de congruence $k_0 \in \N$ tels que $\M_k$ soit un $\O_{\X , \Q}$-module libre de rang $n$ pour tout $k \geq k_0$.
On ne considère maintenant que les indices $k$ supérieurs ou égaux à $k_0$.
On note $\lambda_k : \M_{k+1} \to \M_k$ le morphisme de transition au rang $k$.
Ce dernier est $\widehat{\mathcal{D}}^{(0)}_{\X, k+1 , \Q}$-linéaire donc $\O_{\X , \Q}$-linéaire.
Par hypothèse, l'application $\Dkq \otimes_{\widehat{\mathcal{D}}^{(0)}_{\X, k+1 , \Q}} \M_{k+1} \to \M_k$, $P \otimes e \mapsto P \cdot \lambda_k(e)$ est un isomorphisme $\Dkq$-linéaire.
Cela implique que l'image $\lambda_k(\M_{k+1})$ de $\lambda_k$ est dense dans $\M_k$ pour la topologie $\varpi$-adique.
Comme $\M_k$ est un $\O_{\X , \Q}$-module libre de rang fini, $\lambda_k(\M_{k+1})$ est un sous-$\O_{\X , \Q}$-module fermé de $\M_k$.
Puisqu'il est dense, $\lambda(\M_k) \simeq \M_k$ en tant que $\O_{\X , \Q}$-modules.
Autrement dit, l'application $\lambda_k : \M_{k+1} \to \M_k$ est surjective.
Comme $\M_k$ et $\M_{k+1}$ sont des $\O_{\X , \Q}$-modules libre de même rang fini $n$, $\lambda_k$ est un isomorphisme de $\O_{\X , \Q}$-modules.
On en déduit que $\M \simeq \varprojlim_{k \geq k_0} \M_k \simeq \M_{k_0}$ en tant que $\O_{\X , \Q}$-module.
Ainsi, $\M$ est un $\O_{\X , \Q}$-module libre de rang fini $n$.
\end{proof}

La réciproque de ce lemme est vraie : si $\M = \varprojlim_k \M_k$ est un $\Di$-module coadmissible et un $\O_{\X , \Q}$-module libre de rang $n$, alors il existe un niveau de congruence $k$ à partir duquel chaque $\M_k$ est un $\O_{\X , \Q}$-module libre de rang $n$.
Ce résultat est démontré dans \cite{hallopeau}.

\begin{definition}
Un module coadmissible $\M = \varprojlim_k \M_k$ est appelé module à connexion intégrable s'il existe un rang $k$ à partir duquel chaque $\M_k$ est un $\O_{\X , \Q}$-module libre de rang fini donné $n$.
\end{definition}

Soit $\M = \varprojlim_k \M_k$ un module à connexion de rang $n$.
D'après le lemme \ref{modconnexion} et la proposition \ref{prop3.10}, les modules $\M_k$ ont une unique multiplicité égale à $n$.
On déduit alors de la proposition \ref{propfinitelenght} que tout $\Di$-module intégrable à connexion est de longueur finie.

\begin{prop}
Soit $\M = \varprojlim_k \M_k$ un $\Di$-module à connexion intégrable. Alors $\M$ est un $\Di$-module de longueur finie inférieure ou égale au rang $\mathrm{rg}_{\O_{\X , \Q}}(\M)$.
\end{prop}

\bibliographystyle{plain}
\bibliography{biblio.bib}

\end{document}